\documentclass[11pt,letterpaper]{amsart}
\usepackage{amsmath}
\usepackage{amstext}
\usepackage{amssymb}
\usepackage{amsfonts}
\usepackage{enumerate}
\usepackage{mathrsfs}
\usepackage{braket}
\usepackage{color}
\usepackage[all]{xy}
\usepackage{url}
\usepackage{bbm}
\usepackage{hyperref}

\numberwithin{equation}{section}

\newtheoremstyle{note}
{1em}
{1em}
{}
{}
{\bfseries}
{:}
{.5em}
{}

\newtheorem{theorem}{Theorem}[section]
\newtheorem{lemma}[theorem]{Lemma}
\newtheorem{proposition}[theorem]{Proposition}
\newtheorem{corollary}[theorem]{Corollary}

\theoremstyle{note}
\newtheorem{remark}[theorem]{Remark}

\newtheorem{problem}[theorem]{Problem}
\newtheorem{definition}[theorem]{Definition}

\newtheorem{claim}[theorem]{Claim}
\newtheorem{subclaim}[theorem]{Subclaim}

\newcommand{\fbs}{{\mathscr{F}}}

\newcommand{\dash}{\text{-}}

\newcommand{\N}{{\mathbb{N}}}
\newcommand{\R}{{\mathbb{R}}}
\newcommand{\C}{{\mathbb{C}}}
\newcommand{\Q}{{\mathbb{Q}}}

\newcommand{\K}{{\mathbb{K}}}
\newcommand{\n}[1]{ \left\|#1\right\| }
\newcommand{\tn}[1]{{\left\vert\kern-0.25ex\left\vert\kern-0.25ex\left\vert #1 
    \right\vert\kern-0.25ex\right\vert\kern-0.25ex\right\vert}}

\newcommand{\eps}{\varepsilon}

\newcommand{\Span}{{\mathrm{span}}}
\newcommand{\M}{{\mathrm{M}}}

\newcommand{\cU}{{\mathcal{U}}}
\newcommand{\cF}{{\mathcal{F}}}
\newcommand{\cB}{{\mathcal{B}}}

\DeclareMathOperator{\cb}{cb}
\DeclareMathOperator{\CB}{CB}
\DeclareMathOperator{\Lip}{Lip}

\DeclareMathOperator{\diag}{diag}

\DeclareMathOperator{\MIN}{MIN}
\DeclareMathOperator{\MAX}{MAX}

\title[Lipschitz geometry of operator spaces]{Lipschitz geometry of operator spaces and Lipschitz-free operator spaces}

\author[B. M. Braga]{Bruno M. Braga}
\address[B. M. Braga]{PUC-Rio, Rua Marquês de São Vicente 225, Rio de Janeiro, RJ, Brazil.}

\email{demendoncabraga@gmail.com}
\urladdr{https://sites.google.com/site/demendoncabraga/}

\author[J. A. Ch\'avez-Dom\'inguez]{Javier Alejandro Ch\'avez-Dom\'inguez}
\address[J. A. Ch\'avez-Dom\'inguez]{Department of Mathematics, University of Oklahoma, Norman, OK 73019-3103,
USA} \email{jachavezd@ou.edu}
\urladdr{http://www.math.ou.edu/~jachavezd}

\author[T. Sinclair]{Thomas Sinclair}
\address[T. Sinclair]{Mathematics Department, Purdue University, $150$ N. University Street, West Lafayette, IN $47907$-$2067$}
\email{tsincla@purdue.edu}
\urladdr{http://www.math.purdue.edu/~tsincla/}

\thanks{B. M. Braga was partially supported by NSF grant DMS-2054860.}

\thanks{J. A.  Ch\'avez-Dom\'inguez was partially supported by NSF grant DMS-1900985.}

\thanks{T. Sinclair was partially supported by NSF grants DMS-1600857 and DMS-2055155.}


\begin{document}
\maketitle

\begin{abstract}
We show that there is an operator space notion of Lipschitz embeddability between operator spaces which is strictly weaker than its linear counterpart but which is still strong enough to impose linear restrictions on     operator space structures. This shows that there is a  nontrivial theory of nonlinear geometry for operator spaces and it answers a question in \cite{BragaChavezDominguez2020PAMS}. For that, we  introduce the operator space version of   Lipschitz-free Banach spaces and  prove several properties of it. In particular, we show that separable operator spaces satisfy a sort of isometric Lipschitz-lifting property in the sense of   G. Godefroy and N.  Kalton. Gateaux differentiability of Lipschitz maps in the operator space category is also studied. 
 \end{abstract}

\section{Introduction}\label{SectionIntro}

The study of   nonlinear maps between Banach spaces and how the linear geometry of those spaces are preserved by such maps  dates back to the famous Mazur-Ulam theorem (\cite{MazurUlam1932}). Since then, several different types of nonlinear embeddings and equivalences between Banach spaces have been proven to yield a   rich  theory (we refer the reader to the monographs \cite{BenyaminiLindenstraussBook,OstrovskiiBook2013}). The nonlinear theory becomes particularly interesting when considering  embeddings/equivalences given by maps which are   (1)  coarse, (2)  uniformly continuous, and (3)  Lipschitz.\footnote{Notice that, if $f:X\to Y$ is a map between Banach spaces, the following implications hold: $f$ is linear and bounded $\Rightarrow$ $f$ is Lipschitz  $\Rightarrow$ $f$ is uniformly continuous $\Rightarrow$ $f$ is coarse.}   Besides to functional analysts, this line of research is of interest to theoretical computer scientists (\cite{AilonChazelle2009Siam,BrinkmanCharikar2005}) and operator algebraists working with the Novikov conjecture (\cite{KasparovYu2006}).

This paper concerns the noncommutative counterpart of this theory, the study of which has only recently started to be  investigated (see \cite{BragaChavezDominguez2020PAMS}). For that, it is necessary to find   ``correct'' definitions  of nonlinear morphisms between operator spaces (in order not to extend this introduction too much, we refer the reader to Section \ref{SectionIntro} for the basics of operator spaces). Just as \emph{bounded} maps between Banach spaces give rise to \emph{completely bounded} maps between operator spaces, \emph{coarse} maps also have a natural  	 ``complete version'':    a map $f:X\to Y$ between operator spaces is \emph{completely coarse} if for all $r>0$ there is $s>0$ so that for all $ n\in\N$ the $n$-th amplification $f_n:\M_n(X)\to \M_n(Y)$ satisfies \[ \big\|[x_{ij}]-[y_{ij}]\big\|_{\M_n(X)}\leq r\Rightarrow \big\|f_n([x_{ij}])-f_n([y_{ij}])\big\|_{\M_n(Y)}\leq s\] for all   $[x_{ij}],[y_{ij}]\in \M_n(X)$   
(recall that a map between Banach spaces is \emph{coarse} precisely if the above holds for $n=1$). Although this definition is very natural, the main result of \cite{BragaChavezDominguez2020PAMS} shows that it is not the ``correct'' one. Precisely:\footnote{Throughout this introduction, all operator spaces are considered to be over the complex field. We point out however that all the main results of this paper remain valid for real operator spaces with unchanged proofs.}

 \begin{theorem}\cite[Theorem 1.1]{BragaChavezDominguez2020PAMS}
Let $X$ and $Y$ be  operator spaces, and let $f: X \to Y$ be
completely coarse. If $f(0)= 0$, then $f$ is $\R$-linear.
\label{ThmCompleteCoarseIsRLinear}
\end{theorem}

 Notice that one should have no hope of recovering $\C$-linear maps from nonlinear ones; not even from $\R$-linear ones. Indeed,  there are nonisomorphic $\C$-Banach spaces which are isomorphic as $\R$-Banach spaces (\cite{Bourgain1986PAMS}). Moreover,   there are even $\C$-Banach spaces which are $\R$-linearly isomorphic to each other but totally incomparable as complex spaces  (\cite[Theorem 1]{Ferenczi2007Advances}). So $\R$-linearity is indeed sharp in Theorem  \ref{ThmCompleteCoarseIsRLinear}.

As completely coarse maps are automatically $\R$-affine (Theorem \ref{ThmCompleteCoarseIsRLinear}),  a different kind of nonlinear morphism must be considered in order to develop a (nontrivial)  nonlinear theory for operator spaces.  In order to remedy this situation, \cite{BragaChavezDominguez2020PAMS} proposed the nonlinearization of \emph{almost complete isomorphic embeddings}  instead of the usual \emph{complete isomorphic embeddings}. Precisely:

 \begin{definition}\cite[Definition   4.1]{BragaChavezDominguez2020PAMS}
Let $X$ and $Y$ be operator spaces and $\K\in\{\R,\C\}$. A sequence  $(f^n:X\to Y)_n$ is an \emph{almost complete $\K$-isomorphic embedding} if each $f^n$ is $\K$-linear and there is $K>0$ so that each amplification   $f^n_n:\M_n(X)\to \M_n(Y) $  is a $\K$-isomorphic embedding  with distortion at most $K$.    If $K=1$,   $(f^n:X\to Y)_n$ is an \emph{almost complete $\K$-linearly isometric embedding}.\label{DefiAlmCompIsoEmb}
\end{definition} 

Almost complete isomorphic embeddability is clearly weaker than complete isomorphic embeddability and, by  \cite[Theorem  4.2]{BragaChavezDominguez2020PAMS}, it is actually strictly weaker. Moreover,  Definition \ref{DefiAlmCompIsoEmb} has natural nonlinearizations. However, in order for a nonlinear type of embedding   to be relevant, two conditions must hold: 
\begin{enumerate}[(I)]
\item\label{Item1} the nonlinear embedding  must   still be strong enough  to   recover some linear aspects of the operator space structures of the spaces, and 
\item\label{Item2}   the nonlinear embedding   must be strictly weaker than almost complete isomorphic embeddability.
\end{enumerate}

In this paper, we deal with the following Lipschitz version of Definition \ref{DefiAlmCompIsoEmb}.

 \begin{definition}\label{Defi2}
Let $X$ and $Y$ be operator spaces. A sequence  $(f^n:X\to Y)_n$ is an  \emph{almost complete Lipschitz  embedding} if   there is $K>0$ so that each amplification   $f^n_n:\M_n(X)\to \M_n(Y) $  is a Lipschitz embedding   with distortion at most $K$. If $K=1$, $(f^n:X\to Y)_n$ is an \emph{almost complete isometric embedding}.\label{DefiAlmCompLipEmb}
\end{definition}

It was proved in \cite{BragaChavezDominguez2020PAMS} that, despite its nonlinear nature, the existence of  almost complete Lipschitz  embeddings    imposes restrictions on  the linear operator space  structures of the spaces; hence it satisfies \eqref{Item1} above. For instance, if an infinite dimensional   operator space $X$ almost
completely Lipschitzly embeds into G. Pisier's operator space $\mathrm{OH}$, then $X$ must be   completely  isomorphic
to $\mathrm{OH}$ (see \cite[Theorem 1.2]{BragaChavezDominguez2020PAMS}).  For another known example of how the existence of  almost complete Lipschitz  embeddings  imposes restrictions  to the linear operator space structures, see \cite[Proposition 4.4]{BragaChavezDominguez2020PAMS}.\footnote{ Theorem 1.2 and Proposition 4.4 of \cite{BragaChavezDominguez2020PAMS} are actually  stronger as  they only demand the embeddings to be  \emph{almost completely coarse embedings} (\cite[Definition 4.1]{BragaChavezDominguez2020PAMS}).}

Our main conceptual result shows that almost complete Lipschitz  embeddability  also satisfies  \eqref{Item2} above and therefore it yields a nontrivial theory of nonlinear  geometry of operator spaces; this answers \cite[Question 4.3]{BragaChavezDominguez2020PAMS}. Precisely, we show the following:

 \begin{theorem}\label{ThmAlmCompLipNotAlmLinearCompLip}
There are operator spaces $X$ and $Y$ so that $X$ almost completely isometrically embeds into $Y$ but so that $X$ does not $\R$-isomorphically embed into $Y$. In particular:
\begin{enumerate}
\item\label{ItemThm1} Almost complete isometric embeddability is strictly weaker than almost complete $\R$-linearly isometric  embeddability.
\item\label{ItemThm2} Almost complete Lipschitz embeddability is strictly weaker than almost complete $\R$-isomorphic embeddability.
\end{enumerate}
\end{theorem}

The Banach space $X$ constructed in Theorem \ref{ThmAlmCompLipNotAlmLinearCompLip} is nonseparable and this is actually necessary for its first statement to hold. Precisely, for separable operator spaces, we show the following:

 \begin{theorem}\label{ThmCompIsomEmbImpliesCompLinIsomEmb}
Let $X$ be a separable operator space  and assume that $X$ almost completely isometrically embeds into an operator space $Y$. Then, $X$ almost completely $\R$-isometrically embeds into $Y$.
\end{theorem}

 In particular,  Theorem \ref{ThmCompIsomEmbImpliesCompLinIsomEmb} provides another example of  almost complete  isometric embeddability satisfying \eqref{Item1} above.	We must point out that  Theorem \ref{ThmCompIsomEmbImpliesCompLinIsomEmb} is an operator space version of the result of   G. Godefroy and N.  Kalton which states that if a separable Banach space isometrically embeds into another Banach space, then it does so $\R$-linearly isometrically (see  \cite[Corollary 3.3]{GodefroyKalton2003}).

Using differentiability of Lipschitz maps, we obtain another example of almost complete Lipschitz  embeddability satisfying \eqref{Item1} above. Recall, if a separable Banach space $X$ Lipschitz embeds into a dual Banach space, then it does so $\R$-isomorphically (see \cite[Theorem 3.5]{HeinrichMankiewicz1982}). We prove the following  operator space version of this classic result:

 \begin{theorem}\label{ThmCompLipEmbDualSpImpliesCompEmbBOUND:COROLLARY}
 Let $X$ and $Y$ be operator spaces and assume that $X$ is separable. If $X$ almost completely Lipschitzly embeds into $Y^*$, then $X$ almost completely $\R$-linearly embeds into $Y^*$.\qed
 \end{theorem}

As it is often the case, although Theorem \ref{ThmCompLipEmbDualSpImpliesCompEmbBOUND:COROLLARY} is an operator space version of  \cite[Theorem 3.5]{HeinrichMankiewicz1982}, its proof required  nontrivial adaptations in order for it to hold in  the operator space scenario. Note that by \cite[Theorem 2.9]{Blecher1992} every von Neumann algebra is a dual operator space.
Similarly to the Banach space category, we have that an operator space $Y$ completely linearly isometrically embeds in $Y^{**}$ by \cite[Theorem 2.11]{BlecherPaulsen1991} or \cite[Theorem 2.2]{EffrosRuan1991}; hence, if $X$ almost  completely Lipschitzly embeds in $Y$ it almost completely $\R$-linearly embeds in $Y^{**}$. However, in contrast to the category of Banach spaces, not every operator space is locally reflexive \cite[Chapter 18]{Pisier-OS-book}, so no general Ribe-type theorem, like \cite[Theorem 5.1]{HeinrichMankiewicz1982}, can be deduced from this.

As mentioned above,   Theorem \ref{ThmCompLipEmbDualSpImpliesCompEmbBOUND:COROLLARY} cannot be strengthened in order to obtain   $\C$-linear maps. However, substituting  $Y^*$ by $Y^*\oplus \overline{Y}^*$ --- where $\overline{Y}$ denotes  the conjugate operator space of $Y$ (see Subsection \ref{SubsectionConjugate}) --- we have the following:

 \begin{corollary}\label{CorLipThm}
 Let $X$ and $Y$ be operator spaces and assume that $X$ almost completely Lipschitzly embeds into $Y^*$. Then $X$ almost completely $\C$-linearly embeds into $Y^*\oplus \overline{Y}^*$.\qed
 \end{corollary}

We now give a brief description of    the  methods used in order to obtain our main results.    Recall, given a metric space $(X,d)$, the \emph{Lipschitz-free Banach space of $X$}, denoted by $\cF(X)$, is arguably one of most important  linearization tools (we refer the reader to Section \ref{SectionLipFreeSp} for details). Precisely, $\cF(X)$ is a Banach space so that (1) there is a canonical isometric embedding $\delta_X:X\to \cF(X)$, and (2) given any Lipschitz map $L:X\to Y$ between metric spaces, there is a unique  linear map $\tilde L:\cF(X)\to \cF(Y)$ so that $\|\tilde L\|=\Lip(L)$ and $\tilde L\circ\delta_X=\delta_Y\circ L$ (i.e., $\tilde L$ lifts $L$). 

In Section \ref{SectionLipFreeSp}, we introduce an operator space version of $\cF(X)$. Precisely, for each operator metric space\footnote{As seen in Subsection \ref{SubsectionBasics}, an \emph{operator metric space} $X$  is defined as a subset of $\cB(H)$ for some Hilbert space $H$.} and each $n\in\N$, we define an operator space $\cF^n(X)$ ---   the \emph{$n$-Lipschitz-free operator space of $X$} --- and show that (1) there is a canonical embedding $\delta_X^n:X\to \cF^n(X)$ whose $n$-th amplification is an isometry (see Proposition  \ref{PropDeltaCompIsomIntoFreLipSp}), and (2) given any Lipschitz map $L:X\to Y$ there is a unique linear map $\tilde L:\cF^n(X)\to \cF^n(Y)$ whose completely bounded norm equals the Lipschitz constant of the $n$-th amplification of $L$ and  $\tilde L\circ\delta_X=\delta_Y\circ L$ (see Lemma \ref{LemmaLinearizationLipMaps}). Hence, $\cF^n(X)$ should be seen as the noncommutative version of   $\cF(X)$. Several other properties of $\cF^n(X)$ are proven in Section \ref{SectionLipFreeSp}; for instance, we show that   $\cF^n(X)$ is an $n$-maximal operator space (Remark \ref{RemNLipFreeIsNMaximal}).

In Section \ref{SectionLiftProp}, we introduce an operator space version of the \emph{isometric Lipschitz-lifting property} of   G. Godefroy and N. Kalton  (Definition \ref{DefiLipLiftProp}) and show that every separable operator   space satisfies this property (Theorem \ref{ThmIsoLipLifProp}). Together with an operator space version of an old result of T. Figiel (Proposition \ref{PropFigielComplete}), this is  the main tool in order to obtain Theorem \ref{ThmCompIsomEmbImpliesCompLinIsomEmb}.

We present in   Section \ref{SectionAlternative}  a second approach to Lipschitz-free operator spaces as an alternative to the one presented in Section \ref{SectionLipFreeSp}. In a nutshell, the operator space norms of $\cF^n(X)$ described in Section \ref{SectionLipFreeSp} are given in terms of a \emph{supremum} and in Section \ref{SectionAlternative} we present an alternative description of those norms in terms of an \emph{infimum} (the treatment in this section   follows the presentation of Lipschitz-free spaces given in \cite{Arens-Eells,Weaver1999BookSecondEdition} while our approach for Section \ref{SectionLipFreeSp} follows \cite{GodefroyKalton2003}). Moreover, we use this approach in Subsection \ref{SubsectionExamples} in order to compute $\cF^n(X)$ for some simple operator metric spaces.

At last, Section \ref{SectionDiff} deals with differentiability of Lipschitz maps in the operator space category. Precisely, this section adapts several results of \cite{HeinrichMankiewicz1982} about  Gateaux $w^*$-$\R$-differentiability to the operator spaces setting. The tools in this section allow us to obtain Theorem \ref{ThmCompLipEmbDualSpImpliesCompEmbBOUND:COROLLARY} (cf. \cite[Theorem 3.5]{HeinrichMankiewicz1982}).

\section{Preliminaries}\label{SectionPrelim}

Throughout this paper, $\K$ denotes either $\R$ or $\C$. All Banach spaces are assumed to be over the complex field and all linear maps are assumed to be $\C$-linear unless otherwise stated. In this case, we refer to those as $\R$-Banach spaces, $\R$-operator spaces, $\R$-linear maps, etc. However, we point out that all of our main results are valid for $\R$-operator spaces as well (see \cite{Ruan-real-OS} for a detailed treatment of $\R$-operator spaces). 

For each $n\in\N$, we let $\M_n$ denote the space of $n$-by-$n$ matrices with complex entries and $\|\cdot\|_{\M_n}$ denotes the canonical operator norm on $\M_n$.

\subsection{Basics on operator metric spaces}\label{SubsectionBasics}
Given a  Hilbert space $H$ and $n\in\N$, we denote the space of bounded operators on $H$ endowed with its operator norm by $\cB(H)$ and the $\ell_2$-sum of $n$ copies of $H$ by $H^{\oplus n}$. A subset $X\subset \cB(H)$ is called an \emph{operator metric space}. If $X$ is moreover a closed linear subspace of $\cB(H)$, then $X$ is an \emph{operator space}.

Given an operator metric space $X\subset \cB(H)$ and   $n\in \N$, the matrix space $\M_n(X)$ has a canonical   norm  given by its canonical realization as a   subspace of  $\cB(H^{\oplus n})$.  Elements of $\M_n(X)$ are denoted by $[x_{ij}]$ --- it is implicit that $i,j$ varies among $\{1,\ldots, n\}$.  

Let $f: X\to Y$ be a map between operator metric spaces. For each $n\in\N$, the \emph{$n$-th amplification}  $f_n: \M_n(X)\to \M_n(Y)$ is  defined by letting \[f_n([x_{ij}]) := [f(x_{ij})]\]
for all $[x_{ij}]\in \M_n(X)$.  We say that $f$ is an  \emph{$n$-isometry} if $f_n$ is an isometry and a \emph{complete isometry} if each $f_n$ is an $n$-isometry. If $X$ and $Y$ are operator spaces and $f:X\to Y$ is $\K$-linear, then each $f_n$ is also $\K$-linear and its norm is denoted by $\|f_n\|_n$. In this case, we say that $f$ is an \emph{$n$-contraction} if $f_n$ is a contraction (i.e., if $\|f_n\|_n\leq 1$) and a \emph{complete contraction} if each $f_n$ is an $n$-contraction. Moreover, $f$ is \emph{completely bounded} (abbreviated by \emph{cb}) if 
\[\|f\|_{\cb}= \sup_n \|f_n\|_n<\infty\]
 and $f$ is a \emph{complete $\K$-isomorphic embedding} if both $f$ and $f^{-1}$ are completely bounded.

The space of all linear cb-maps  $X\to Y$ between operator spaces is denoted by  $\CB(X,Y)$ and $\|\cdot\|_{\cb}$ defines a complete norm on $\CB(X,Y)$. The norm $\|\cdot\|_{\mathrm cb}$ is called the \emph{cb-norm}. Moreover,  $\CB(X,Y)$  carries a natural operator space structure itself. Precisely, given $n\in\N$, the matrix norm on $\M_n(\CB(X,Y))$ is given by the canonical isomorphism  $\M_n(\CB(X,Y))\cong\CB(X,\M_n(Y))$.  In particular, the Banach space dual of an operator space $X$ is an operator space via the norms inherited from the identifications $\M_n(X^*)\cong \CB(X,\M_n)$.

It is well known that for a given Banach space $X$, among all the  possible operator space structures which are compatible with the norm of $X$ there are a smallest one and a largest one, denoted $\MIN(X)$ and $\MAX(X)$ respectively (see \cite[Chap. 3]{Pisier-OS-book}).
More generally, following \cite{OikhbergRicard2004MathAnn} given an operator space $X$ and a natural number $n$ we will consider the smallest and largest operator space structures that are compatible with the norm on $\M_n(X)$, which will be denoted by $\MIN_n(X)$ and $\MAX_n(X)$ respectively. Explicit descriptions can be found in \cite[Definition I.3.2]{LehnerPhDThesis} or \cite[Section 2]{OikhbergRicard2004MathAnn}, though we essentially will not need them, so, for our purposes, the following property can be taken as their definition \cite[Lemma 2.3]{OikhbergRicard2004MathAnn}: given an operator space $X$, $\MIN_n(X)$ (resp. $\MAX_n(X)$) is the unique operator structure on $X$ which agrees with that of $X$ up to the $n$-th matricial level, and such that for any operator space $Y$ and any linear map $u : Y \to X$ (resp. $v : X \to Y$) we have
$\|u : Y \to \MIN_n(X)\|_{\cb} = \| u_n : \M_n(Y) \to \M_n(X) \|_n$
(resp. $\|v : \MAX_n(X) \to Y\|_{\cb} = \| v_n : \M_n(X) \to \M_n(Y) \|_n$).

\subsection{Almost complete isomorphic embeddings}\label{SubsectionAlmCompIsoEmb} The next definition is a weakening of complete isomorphic/isometric embeddings and it was introduced in \cite[Definition 4.1]{BragaChavezDominguez2020PAMS}. Moreover, as shown in \cite[Theorem 4.2]{BragaChavezDominguez2020PAMS}, this is a strict weakening. 
 
\begin{definition}[Definition \ref{DefiAlmCompIsoEmb}]
Let $X$ and $Y$ be operator metric space and consider a sequence of $\K$-linear maps $(f^n:X\to Y)_n$.
\begin{enumerate}
\item The sequence $(f^n)_n$ is an \emph{almost    complete  $\K$-linear isometric embedding of $X$ into $Y$} if the $n$-th amplification of each $f^n$ is an isometry. In this case,   \emph{$X$ almost   completely $\K$-linear isometrically embeds into $Y$}.
\item The sequence $(f^n)_n$ is an \emph{almost complete $\K$-isomorphic embedding of $X$ into $Y$} if there is $D>0$ so that the  $n$-th amplification of each $f^n$ is a $D$-isomorphism.   In this case,   \emph{$X$ almost   completely   $\K$-isomorphically embeds into $Y$}.
\end{enumerate}
If $\K$ is not specified, it is always assumed to be $\C$.\label{DefiAlmostCompLinearLipEmb}
\end{definition}

Unlike the category of separable Banach spaces, there is no linearly isometrically universal element in the category of separable operator spaces, i.e., there is no separable operator space $X$ so that all separable operator spaces can be linearly isometrically embedded into $X$ (see \cite[Section 2.12]{Pisier-OS-book}).  As we show in the next proposition, this is no longer the case for almost   complete linear isometric embeddings. For that, let $\Delta$ denote the Cantor set  $\{0,1\}^\N$.   Given $n\in\N$, $C(\Delta,\M_n)$ denotes the Banach space of all continuous functions $\Delta\to \M_n$ endowed with the supremum norm. We view $C(\Delta,\M_n)$ with the canonical  operator space structure determined by the fact that it is a $C^*$-algebra.

\begin{proposition}\label{UniversalityForLinearnIsometricEmbeddings}
Let $X$ be a separable operator space and $n\in\N$. Then there is an $n$-isometry $X\to C(\Delta,\M_n)$.  In particular, the separable operator space $\big(\bigoplus_{n\in\N} C(\Delta,\M_n) \big)_{c_0}$ is universal for almost complete linear isometric embeddings of separable operator spaces.
\end{proposition} 

\begin{proof}
By \cite[Theorem I.1.9]{LehnerPhDThesis}, there is a linear $n$-isometric embedding $f^n: X \to C(K_n,\M_n)$ where $K_n$ is the unit ball of $\M_n(X^*)$ endowed with the weak$^*$ topology.
Since $X$ is separable, $K$ is metrizable. So,   by the Alexandroff--Urysohn theorem, there is a continuous surjection $q_n : \Delta \to K_n$.
Finally, observe that the map $C(K_n,\M_n) \to C(\Delta, \M_n)$ given by $f^n \mapsto f^n \circ q_n$ is a complete isometry.

The maps $(f^n\circ q_n)_{n\in \mathbb N}$ clearly induce an almost complete linear isometric embedding of $X$ into $\big(\bigoplus_{n\in\N} C(\Delta,\M_n) \big)_{c_0}$. \qedhere

\end{proof}

\subsection{Almost complete Lipschitz embeddings}\label{SubsectionAlmCompLipEmb}
As mentioned in the introduction, although the natural nonlinear versions of complete isomorphic embeddings do not lead to an interesting nonlinear theory for operator spaces (Theorem \ref{ThmCompleteCoarseIsRLinear}), we will show that the natural nonlinearizations of almost complete isomorphic embeddings do (Theorem \ref{ThmAlmCompLipNotAlmLinearCompLip}). We point out that this nonlinearization was first introduced in \cite[Definition 4.1]{BragaChavezDominguez2020PAMS} for the coarse category.

Given metric spaces $(X,d)$ and $(Y,\partial)$, and a map $f:X\to Y$,  we denote the  \emph{modulus of uniform continuity of $f$} by $\omega_f:[0,\infty)\to [0,\infty]$,  i.e., 
 \[\omega_f(t)=\sup\Big\{\partial(f(x),f(y))\mid d(x,y)\leq t\Big\}\]
 for $t\geq 0$. The \emph{Lipschitz constant of $f$} is given by $\Lip(f)=\sup_{t>0}\omega_f(t)/t$ and $f$ is called \emph{Lipschitz} if $\Lip(f)<\infty$. Moreover, if $f$ is an injective  Lipschitz map and $f^{-1}$ is also Lipschitz, then $f$ is a \emph{Lipschitz embedding}. The infimum of all $D>0$ so that there is $r>0$ such that 
 \[r\cdot d(x,y)\leq \partial(f(x),f(y))\leq Dr\cdot d(x,y)\]
 for all $x,y\in X$ is called the \emph{Lipschitz distortion of $f$} and denoted by $\mathrm{Dist}(f)$.
 
 \begin{definition}
 Let   $X$ and $Y$ be operator metric spaces and let $n\in\N$. 
 \begin{enumerate}
     \item The \emph{$n$-th Lipschitz constant of $f$} is  given by  $\Lip_n(f)=\sup_{t>0}\omega_{f_n}(t)/t$. Equivalently,  $\Lip_n(f)= \Lip(f_n)$.
     \item If $f$ is a Lipschitz embedding, the \emph{$n$-th Lipschitz distortion} of $f$ is defined as $\mathrm{Dist}_n(f)=\mathrm{Dist}(f_n)$. 
 \end{enumerate}
 \end{definition}
 
 Notice that, if $f$ is Lipschitz, then $f_n$ is   Lipschitz for all $n\in\N$. However, if $f$ is not $\R$-affine,\footnote{A map $f:X\to Y$ between $\K$-vectors spaces is called $\K$-affine if $g=f-f(0)$ is $\K$-linear.} then $\sup_{n\in\N}\Lip_n(f)=\infty$ by Theorem \ref{ThmCompleteCoarseIsRLinear}.  In order to overcome this issue, we look at \emph{almost complete Lipschitz embeddings} (cf. Definition \ref{DefiAlmostCompLinearLipEmb}):

\begin{definition}[Definition \ref{Defi2}] 
Let $X$ and $Y$ be operator metric spaces and consider a sequence of maps    $(f^n:X\to Y)_n$.
\begin{enumerate}
\item The sequence $(f^n)_n$ is an \emph{almost complete isometric  embedding of $X$ into $Y$} if each $n$-th amplification of $f^{n}$ is an isometry.  In this case, we say that \emph{$X$ almost completely isometrically  embeds into $Y$}.
\item The sequence $(f^n)_n$ is an \emph{almost complete Lipschitz embedding of $X$ into $Y$} if there is $D>0$ so that $\mathrm{Dist}_n(f^n)\leq D$ for all $n\in\N$. 
 In this case, we say that \emph{$X$ almost completely Lipschitzly embeds into $Y$}.
\end{enumerate} \label{DefiAlmostCompLipEmb}
\end{definition} 
 
 \subsection{Obtaining  $\C$-linear maps from $\R$-linear maps}\label{SubsectionConjugate}
 As mentioned in the introduction, it is not always possible to recover $\C$-linear maps from $\R$-linear maps; not even in the Banach space category (see \cite{Bourgain1986PAMS,Kalton1995CanBull,Ferenczi2007Advances}). However, $\R$-linear embeddability  still sheds some  light on $\C$-linear embeddability; at least  if one is  allowed to change   the target space ``slightly''. Before stating the precise result, we recall the concept of conjugate operator space. 
 
 Given a $\C$-Banach space $X$, we denote   the \emph{conjugate    of $X$} by $\overline{X}$, i.e., $\overline {X}=X$ as a set and the scalar multiplication on $\overline{X}$ is given by $\alpha x=\bar \alpha x$ for all $\alpha \in \C$ and all $x\in \overline{X}$. Then, given a $\C$-operator space $Y\subset B(H)$, $\overline{Y}$ denotes the \emph{conjugate operator space of $Y$}, i.e.,  $\overline{Y}=Y$ and the operator space structure on $Y$ is given by the canonical inclusion $\overline{Y}\subset \overline{B(H)}= B(\overline{H})$.

 The following simple proposition can be obtained just as \cite[Proposition 4.5]{BragaChavezDominguez2020PAMS}, so we omit the details.

\begin{proposition}\label{PropYoplusYBar}
If a $\C$-operator space $X$ almost completely $\R$-isomorphically embeds into a $\C$-operator space $Y$, then $ X$ almost completely $\C$-isomorphically embeds into $Y \oplus \overline{Y}$.
\end{proposition}

 \section{$n$-Lipschitz-free operator spaces}\label{SectionLipFreeSp}
 
Given a metric space $X$, one can assign to it a Banach space $\cF(X)$ which is called the \emph{Lipschitz-free (Banach) space of $X$}. This construction comes equipped with  a canonical isometry $\delta_X:X\to \cF(X)$ and  one of its main features is that for any Lipschitz map $L:X\to Y$ there is   an unique linear map $\tilde L:\cF(X)\to \cF(Y)$ so that $\|\tilde L\|=\Lip(L)$ and $\tilde L \circ \delta_X=\delta_Y\circ L$. This linearization process makes the Lipschitz-free spaces important tools when working with the nonlinear geometry of Banach spaces. In this section, we introduce the  operator space version of Lipschitz-free spaces and use it to show that there is a nontrivial theory of nonlinear geometry for operator spaces (see Theorem \ref{ThmAlmCompLipNotAlmLinearCompLip}). We refer to \cite{Arens-Eells,Weaver1999BookSecondEdition,GodefroyKalton2003,Godefroy2015Survey}
for detailed treatments of Lipschitz-free spaces. 
 
 A pair $(X,x_0)$, where $X$ is an operator metric space and $x_0\in X$, is called a \emph{pointed  operator metric space}. If $X$ is an operator space, we always view it as  a pointed operator metric space with the distinguished point $0\in X$. Given pointed operator metric spaces $(X,x_0)$ and $(Y,y_0)$,   let $\Lip_0(X,Y)$ be the set of all Lipschitz maps $f:X\to Y$ so that $f(x_0)=y_0$. If   $Y$ is an   operator   space,  for each $n\in\N$ we define a norm $\|\cdot \|_{\Lip,n}$ on  $\Lip_0(X,Y)$ by letting 
\[\|f\|_{\Lip,n}=\Lip_n(f).\]
If $n=1$, we write $\|\cdot\|_{\Lip}=\|\cdot\|_{\Lip,n}$
The norm $\|\cdot \|_{\Lip,n}$ is a Banach norm on  $\Lip_0(X,Y)$. Notice that
\begin{equation}
\label{Eq.LipLipnEquiv}
\Lip(f)\leq \Lip_n(f)\leq n^2\Lip(f)
\end{equation} 
for all $f\in \Lip_0(X,Y)$; so 
the norms $\|\cdot \|_{\Lip}$ and $\|\cdot \|_{\Lip,n}$ are equivalent. 

The canonical algebraic isomorphisms  \[\M_k(\Lip_0(X,Y))\cong \Lip_0(X,\M_k(Y))\ \text{ for }\ k\in\N\] induce a natural operator space structure on the Banach space $(\Lip_0(X,Y),\|\cdot\|_{\Lip,n})$.   Precisely,   for $k\in\N$ and   $[f_{ij}]\in \M_k(\Lip_0(X,Y))$, we let   \[\|[f_{ij}]\|_{\Lip,n,k}=\Lip_n\Big([f_{ij}]: X\to \M_k(Y)\Big).\] 

\begin{definition}
Let $(X,x_0)$ be a pointed operator metric space and  let $Y$ be an   operator  space. We denote the operator space   $(\Lip_0(X,Y),(\|\cdot\|_{\Lip,n,k})_{k\in\N})$ defined above by $\Lip^n_0(X,Y)$.
\end{definition}

\begin{remark}\label{remark-Lip^n_0-is-n-minimal}
Given a set $Z$, let $[Z]^2=\{(x,y)\in Z^2\mid x\neq y\}$. Then   the  Lipschitz  norm on $\Lip_0(X,\C)$ can be seen as the norm inherited by the embedding   \[f\in \Lip_0(X,\C)\mapsto \Big(\frac{f(x)-f(y)}{\|x-y\|}\Big)_{(x,y)\in [X]^2} \in \ell_\infty([X]^2,\C).\]
Similarly, given $n\in\N$,   the operator space structure on $\Lip^n_0(X,\C)$ is given by
the  embedding 
\[f\in \Lip_0(X,\C)\mapsto \Big(\frac{f_n(x)-f_n(y)}{\|x-y\|_{\M_n(X)}}\Big)_{(x,y)\in [\M_n(X)]^2}\in \ell_\infty([\M_n(X)]^2,\M_n).\]  
Therefore it follows from Smith's lemma \cite[Proposition 1.12]{Pisier-OS-book} that the $\cb$-norm of any linear map with values on $\Lip^n_0(X,\C)$ is equal to the norm of its $n$-th amplification, and thus $\Lip^n_0(X,\C)$ is an 
is an $n$-minimal operator space.
\end{remark}

The following proposition is straightforward, so we omit its proof.

\begin{proposition}\label{PropMapIn}
Let $H$ be a Hilbert space and $(X,x_0)$ be a pointed operator metric space with $  X	\subset \cB(H)$ and let $n\in\N$. The canonical map 
\[\iota:  \cB(H)^*\to \Lip_0^n(X,\C)\] given by $\iota(a)=a\restriction X-a(x_0)$, for all  $a\in \cB(H)^*$, 
is a complete contraction. \qed
\end{proposition}

Let $(X,x_0)$ be a pointed operator metric space.  Given $x\in X$, define a map \[\delta_x: \Lip_0(X,\C)\to \C\] by letting \[\delta_x(f)=f(x)\ \text{ for all }\ f\in \Lip_0(X,\C).\] So $\delta_x\in \Lip_0(X,\C)^*$ for all $x\in X$. Notice that, given $n\in\N$, as $\Lip^n_0(X,\C)$ is an operator space, so is its dual $\Lip_0^n(X,\C)^*$. So we can   write  $\delta_x\in \Lip_0^n(X,\C)^*$ for all $x\in X$ and all $n\in\N$. 
 
\begin{definition}\label{DefNLipFree}
Let $(X,x_0)$  be a pointed operator metric space  and let $n\in \N$. We define the  \emph{$n$-Lipschitz-free operator space of $(X,x_0)$} as the Banach space 
\[\cF^n(X)=\overline{\Span}\{\delta_x\in \Lip^n_0(X,\C)^*\mid x\in X\}\]
together with the operator space structure inherited from $\Lip^n_0(X,\C)^*$. If $n=1$, we write $\cF(X)=\cF^1(X)$. 
\end{definition}

In the purely metric setting (i.e., when no operator structure is assumed), given a metric space $X$ , the  \emph{Lipschitz-free Banach space of $X$} is the Banach space $\cF(X)$. In Banach space theory,  $\cF(X)$ is usually defined with respect to real-valued Lipschitz maps $X\to \R$ with $f(x_0)=0$; see Subsection \ref{SubsectionRealLipFree} for more on that. 
 
 \begin{remark}
 Notice that the operator space structure of $\cF(X)=\cF^1(X)$ is generally not the one of $\mathrm{MIN}(\cF(X))$.
 First, by Remark \ref{remark-Lip^n_0-is-n-minimal}, $\Lip_0(X,\C)$ is minimal and so $\Lip_0(X,\C)^*$ is maximal,
which implies that $\cF^1(X)$ is \emph{submaximal} (i.e., it is a subspace of a maximal operator space).
In fact, more is true: by Remark \ref{RemNLipFreeIsNMaximal} below, $\cF^1(X)$ is itself maximal.
Since it is well-known that for an operator space of dimension at least 3 its minimal and maximal operator space structures are different (see the discussion after \cite[Theorem 14.3]{Paulsen2002}), we conclude that the operator space structure we have defined on $\cF(X)$ is not that of $\mathrm{MIN}(\cF(X))$ whenever $|X| > 3$ (and thus $\dim(\cF^1(X)) \ge 3$).
Moreover, this is sharp. It is clear that when $|X| \le 2$ the space $\cF(X)$ has a unique operator space structure, and the same can happen when $|X|=3$: if we take $X$ to be a path graph of length 2 endowed with the shortest path metric, $\cF(X)$ is isometric to $\ell_1^2$ and this space also has a unique operator space structure \cite[Page 190]{Paulsen2002}.
 \end{remark}

\begin{proposition}\label{PropIsoLipFreeAndNLipFree}
Let $(X,x_0)$  be a pointed operator metric space. Then, for each $n\in\N$,  $\cF(X)=\cF^n(X)$ as a set and the identity $\cF(X)\to \cF^n(X)$ is a complete isomorphism. 
\end{proposition}

\begin{proof}
 As noticed above , the norms $\|\cdot \|_{\Lip}$ and $\|\cdot \|_{\Lip,n}$ are equivalent. Moreover, the  same argument gives that  $\|\cdot \|_{\Lip,1,k}$ and $\|\cdot \|_{\Lip,n,k}$ are $n^2$-equivalent for all $k$. 
 Hence, the adjoint of the identity map $(\Lip_0(X,\C),\|\cdot\|_{\Lip})\to(\Lip_0^n(X,\C),\|\cdot\|_{\Lip,n}) $ is a complete  isomorphism between $\Lip^n_0(X,\C)^*$ and $\Lip_0(X,\C)^*$. 
 So its restriction to $\cF^n(X)$ is a complete isomorphism between $\cF^n(X)$ and $\cF(X)$.
\end{proof}

We now show that the basic properties of Lipschitz-free Banach spaces have operator space versions. 
 
\begin{proposition}\label{PropDeltaCompIsomIntoFreLipSp}
Let $(X,x_0)$ be a pointed operator metric space and let $n\in\N$. The map \[\delta_{X}^n:x\in X\mapsto\delta_x\in \cF^n(X)\] is an $n$-isometric embedding.
\end{proposition}

\begin{proof}
Let $\delta=\delta_{X}^n$ and fix $[x_{ij}],[y_{ij}]  \in \M_n(X)$. So \[\delta_n([x_{ij}])-\delta_n([y_{ij}]) =[\delta_{x_{ij}}  -\delta_{y_{ij}}  ]\in \M_n(\cF^{n}(X))\subset \CB(\Lip_0^{n}(X,\C),\M_n).\]
Let $k\in\N$ and $[f_{m\ell}]\in \M_k(\Lip_0^{n}(X,\C))$ with $\|[f_{m\ell}]\|_{\Lip,n,k}\leq 1$. As  \[[\delta_{x_{ij}}  -\delta_{y_{ij}}  ]_n([f_{m\ell}])=[f_{m\ell}(x_{ij})-f_{m\ell}(y_{ij})], \] it follows from the definition of the norm $\|\cdot\|_{\Lip,n,k}$ that 
\[\|[\delta_{x_{ij}} -\delta_{y_{ij}}  ]_n([f_{m\ell}])\|_{\M_{nk}} \leq \|[x_{ij}-y_{ij}]\|_{\M_n(X)}.\]
Hence, $\|\delta_n([x_{ij}])-\delta_n([y_{ij}])\|_{\M_n(\cF^{n}(X))}\leq\|[x_{ij}-y_{ij}]\|_{\M_n(X)}$.

For the reverse inequality, fix a Hilbert space $H$ so that $ X\subset \cB(H)$ as an operator metric space. Fix $\eps>0$ and pick unit vectors $\bar\xi=(\xi_i)_i,\bar\zeta=(\zeta_i)_i\in H^{\oplus n}$ such that \[\Big|\sum_{i,j}\langle (x_{ij}-y_{ij})\xi_j,\zeta_i\rangle\Big|=\big|\langle[x_{ij}-y_{ij} ]\bar\xi,\bar\zeta\rangle\big|\geq \|[x_{ij}-y_{ij} ]\|_{\M_n(X)}-\eps.\] 
Let $K=\Span\{\xi_1,\ldots,\xi_n,\zeta_1,\ldots,\zeta_n\}$, let $P_K:H\to K$ be the orthogonal projection onto $K$, and let $g:\cB(H) \to \cB(K)$ be the map given by $g(x)=P_Kx\restriction K$ for all $x\in \cB(H)$. So $g$ is a completely contractive linear map.  As  $K$ has finite dimension, there is  no loss of generality to assume that $\cB(K)=\M_k$ for $k=\dim(K)$.

Let $\iota:  \cB(H)^*\to \Lip_0^n(X,\C)$ be given by Proposition \ref{PropMapIn}. As $g:\cB(H)\to \M_k$ is a complete contraction,   $g$ is in the unit ball of $ \M_k(\cB(H)^*)$; so Proposition \ref{PropMapIn} implies that $\iota_k(g)$ is in the unit ball of $\M_k(\Lip_0^n(X,\C))$.  Hence, we have 
\begin{align*}
\|[\delta_{x_{ij}}-\delta_{y_{ij}}]\|_{\M_n(\cF^{n}(X))}&=\|[\delta_{x_{ij}}-\delta_{y_{ij}}]\|_{\CB(\Lip_0^n(X,\C),\M_n)}\\
&\geq \|[\delta_{x_{ij}}-\delta_{y_{ij}}]_n(\iota_k(g))\|_{\M_{nk}}\\
&=\|[\iota_k(g)(x_{ij})-\iota_k(g)(y_{ij})]\|_{\M_{nk}}\\
&=\|[g(x_{ij}- y_{ij})]\|_{\M_{nk}}.
\end{align*}
Therefore, by definition of $g$, we conclude that 
\[\|[\delta_{x_{ij}}-\delta_{y_{ij}}]\|_{\M_n(\cF^{n}(X))}\geq  \|[x_{ij} ]-[y_{ij} ]\|_{\M_n(X)}-\eps .\]
As $\eps$ was arbitrary, this shows  that \[\|\delta_n([x_{ij}])-\delta_n([y_{ij}])\|_{\M_n(\cF^{n}(X))}\geq \|[x_{ij}]-[y_{ij}]\|_{\M_n(X)}\] and the result follows.
\end{proof}

The following corollary follows straightforwardly from Proposition 
\ref{PropDeltaCompIsomIntoFreLipSp}.

\begin{corollary}\label{CorACLEmbSumLipFree}
Every  pointed operator  metric space $(X,x_0)$    almost completely isometrically embeds into $\big(\bigoplus_{n\in\N}\cF^n(X)\big)_{c_0}$.\qed
\end{corollary}

Let $(X_\lambda)_{\lambda\in \Lambda}$ be a family   of  operator metric spaces and $\cU$ be an ultrafilter on  $\Lambda$. The ultraproduct of  $(X_\lambda)_{\lambda\in \Lambda}$ with respect to $\cU$ is denoted by $(\prod_{\lambda\in \Lambda}X_\lambda)/\cU$ and its elements, i.e., equivalence classes of elements in $(\prod_{\lambda\in \Lambda}X_\lambda)$, are denoted by $[(x(\lambda))_\lambda]$.  The ultraproduct $Y=(\prod_{\lambda\in \Lambda}X_\lambda)/\cU$ has a canonical operator space structure so that if  $n\in\N$ and $[x_{ij}]\in \M_n(Y)$, then \[\|[x_{ij}]\|_{\M_n(Y)}=\lim_{\lambda,\cU} \|[x_{ij}(\lambda)]\|_{\M_n(X)},\]
where $x_{ij}=[(x_{ij}(\lambda))_\lambda]\in Y$ for each $i,j\in \{1,\ldots,n\}$.
In particular, it easily follows that the inclusion $I_{X}:X\to Y$ given by $I_{X}(x)=[(x)_\lambda]$ is  a complete isometry into a subspace of $Y$. If $\Lambda=\N$ and $(X_n)_{n\in\N}$ is a constant family, say $X=X_n$, we denote $(\prod_{\lambda\in \Lambda}X_\lambda)/\cU$  by $X^\cU$. 

The next corollary also follows  straightforwardly from Proposition 
\ref{PropDeltaCompIsomIntoFreLipSp}. The reader can see a similar computation in   \cite[Proposition 4.4]{BragaChavezDominguez2020PAMS}.

\begin{corollary}\label{CorCompEmbIntoUltraPowerLipFree}
Let  $(X,x_0)$ be a pointed operator metric space and $\cU$ be a nonprincipal ultrafilter on $\N$. Then \[x\in X\mapsto \big[(\delta_X^n(x))\big]_n\in    \Big(\prod_n \cF^n(X)\Big)/\cU\] is a complete isometric embedding.\qed
\end{corollary}

We can now show that  the notion of almost complete Lipschitz embeddability is a truly nonlinear notion; in particular,  this  solves \cite[Question 4.3]{BragaChavezDominguez2020PAMS}.

\begin{proof}[Proof of Theorem \ref{ThmAlmCompLipNotAlmLinearCompLip}]
Let $H$ be any  nonseparable Hilbertian operator space.   By Corollary \ref{CorACLEmbSumLipFree}, $H$ almost completely isometrically embeds into \[Y=\big(\bigoplus_{n\in\N}\cF^n(H)\big)_{c_0}.\] Hence, we only need to show that  $H$ does not $\R$-isomorphically embed into $\big(\bigoplus_{n\in\N}\cF^n(H)\big)_{c_0}$. Items \eqref{ItemThm1} and \eqref{ItemThm2} will then follow immediately.

 For the remainder  of this proof,
we consider each $\cF^n(H)$ as an $\R$-Banach space and suppose $T:H\to Y$ is an $\R$-isomorphic embedding. By a sliding hump argument, for all $\eps>0$ there is $k\in\N$ and a finite codimensional subspace $H'\subset H$ such that $\|P_n\circ T\restriction H'\|\leq \eps$, where $P_n:Y\to (\bigoplus _{n>k}\cF^n(H))_{c_0}$ is the canonical projection. Choosing $\eps$ small enough and letting $Q_n=\mathrm{Id}_Y-P_n$, this gives that $Q_n\circ T\restriction H'$ is an isomorphic embedding of $H'$ into $\bigoplus _{n=1}^k\cF^n(H) $. As $H'$ has finite codimension, it is isomorphic to $H$. Therefore, this shows that $H$ isomorphically embeds into $  \bigoplus _{n=1}^k\cF^n(H) $. Since each $\cF^n(H)$ is isomorphic to $\cF(H)$ (Proposition \ref{PropIsoLipFreeAndNLipFree}), we conclude that $H$ isomorphically embeds into $ \bigoplus _{n=1}^k\cF(H)$.

By \cite[Theorem 3.1]{Kaufmann2015Studia}, $\cF(H)$ is isomorphic to $ (\bigoplus _{n=1}^\infty\cF(H))_{\ell_1}$. Therefore,    $  \bigoplus _{n=1}^k\cF(H) $ embeds isomorphically into $\cF(H)$ (in fact, \cite[Theorem 3.1]{Kaufmann2015Studia} combined with Pelczynski's decomposition technique gives us that these two spaces are isomorphic). Therefore, by the previous paragraph, we have that $H$ isomorphically embeds into $\cF(H)$. However, \cite[Theorem 5.21]{Weaver1999BookSecondEdition} says that, as $H$ is nonseparable,  this cannot happen; a contradiction. \qedhere

\end{proof}

 Another important property of the Banach space $\cF(X)$ is that it is an isometric predual for $\Lip_0(X,\C)$. The next proposition gives the operator space version of this result.

\begin{proposition}\label{prop-duality-for-F^n}
Let $(X,x_0)$ be a pointed operator metric space  and $n\in\N$. Then $\cF^{n}(X)^*$ is completely isometric to $\Lip_0^{n}(X,\C)$. Moreover, under this complete isometry, the weak$^*$ topology on $\cF^{n}(X)^*$ coincides on the unit ball of $\Lip_0^{n}(X,\C)$ with the pointwise convergence topology.
\end{proposition}

\begin{proof}
We define a map $u:\Lip_0^{n}(X,\C)\to \cF^{n}(X)^*$ by letting \[u(f)\Big(\sum_ia_i\delta_{x_i} \Big)=\sum_ia_if(x_i)\] for all   $a_1,\ldots, a_m\in\C$, all $x_1,\ldots, x_m\in X$, and all $f\in \Lip_0^{n}(X,\C)$.
Clearly, $u(f)$ is linear on the span of $\{\delta_x\}_{x\in X}$. Moreover, 
 we have that \[\|u(f)\|_{\cF^{n}(X)^*}=\sup\Big|\sum_ia_if(x_i)\Big|=\sup\Big|\Big(\sum_ia_i\delta_{x_i} \Big)(f)\Big|\leq \|f\|_{\Lip,n},\]
where the suprema above are taken over all $a_1,\ldots, a_m\in\C$ and $x_1,\ldots, x_m\in X$ such that $\|\sum_ia_i\delta_{x_i} \|_{\cF^{n}(X)}\leq 1$. This implies that  $u(f)$ extends uniquely to a functional on $\cF^{n}(X)$ with norm at most $\|f\|_{\Lip,n}$; so $u$ is a  well defined contraction. Let $v:\cF^{n}(X)^*\to \Lip_0^{n}(X,\C)$ be given by $v(g)(x)=g(\delta_x )$ for all $g\in \cF^{n}(X)^*$ and all $ x\in X$. By Proposition \ref{PropDeltaCompIsomIntoFreLipSp}, $v(g)$ is indeed Lipschitz, so $v$ is   well defined. Moreover,   it is straightforward to check that $v$ is the inverse of $u$. Hence, we only need to check that $u$ is a $k$-isometry for all $k\in\N$.

Fix $m\in\N$. The fact that $u$ is an  $m$-contraction follows completely analogously to the proof that $u$ is a contraction, so we omit the details. Now let  $[f_{k\ell }]\in \M_m(\Lip_0^n (X,\C))$ and $\eps>0$, and pick   $[x_{ij}],[y_{ij}]\in \M_n(X)$ such that 
\[\|[f_{k\ell }]\|_{\Lip,n,m}\leq \frac{\|[f_{k\ell}(x_{ij})-f_{k\ell}(y_{ij})]\|_{\M_{nm}}}{\|[x_{ij}-y_{ij}]\|_{\M_n(X)}} +\eps.\]  By Proposition \ref{PropDeltaCompIsomIntoFreLipSp}, $\|[x_{ij}-y_{ij}]\|_{\M_n(X)}=\|[\delta_{x_{ij}}-\delta_{y_{ij}}]\|_{\M_n(\cF^n(X))}$, and  we have that 
\begin{align*}
 \Big\|\Big[\frac{f_{k\ell}(x_{ij})-f_{k\ell}(y_{ij})}{\|[x_{ij}-y_{ij}]\|_{\M_{n}(X)}}\Big]\Big\|_{\M_{nm}}
&=\Big\|\Big[u(f_{k\ell})\Big(\frac{\delta_{x_{ij}}-\delta_{y_{ij}}}{\|[\delta_{x_{ij}}-\delta_{y_{ij}}]\|_{\M_n(\cF^n(X))}}\Big)\Big]\Big\|_{\M_{nm}}\\
&\leq \|u_m([f_{k\ell}])\|_{\M_m(\cF^n(X)^*)}.    
\end{align*}
Since $\eps>0$ was arbitrary, this shows that $\|u_n([f_{k\ell}])\|_{\M_m(\cF^n(X)^*)}=\|f\|_{\Lip,n,m}$, so $u$ is an $m$-isometry. 

The last statement in the proposition follows immediately form the formula of the complete isometry presented above.
\end{proof}

\begin{remark}\label{RemNLipFreeIsNMaximal}
The previous result shows that $\mathcal F^n(X)$ is the predual of  an $n$-minimal operator space, which implies that $\mathcal F^n(X)$ is $n$-maximal. To see this, note that if $E$ is an $n$-minimal operator space, then $E^*$ is $n$-maximal by \cite[Lemma 2.4]{OikhbergRicard2004MathAnn}. Thus $\mathcal F^n(X)^{**}$ is $n$-maximal. Now, for any operator spaces $E$ and $F$ and any linear map $v: E\to F$, we have that $\|v^{**}\|_{cb} = \|v\|_{cb}$ by the proof of \cite[Theorem 2.5]{Blecher1992}; hence, $\mathcal F^n(X)$ is itself $n$-maximal.
\end{remark}

The next lemma is one of the main benefits of the Lipschitz-free construction and it works as a tool to linearize certain problems (cf. \cite[Lemma 2.2]{GodefroyKalton2003}).

\begin{lemma}\label{LemmaLinearizationLipMaps}
Let $(X,x_0)$ and $(Y,y_0)$ be pointed operator metric spaces and let $n\in\N$. For any  $L\in \Lip_0(X,Y)$, the $\cb$-norm of the map 
\[f\in \Lip_0^{n}(Y,\C)\mapsto f\circ L\in \Lip_0^{n}(X,\C)\]
equals $\Lip_n(L)$. Moreover,   there exists a unique  bounded linear map $\tilde L:\cF^{n}(X)\to \cF^{n}(Y)$ such that $\delta^n_{Y}L=\tilde L\delta^n_{X}$, i.e., the diagram below commutes.
\[\xymatrix{X\ar[r]^L\ar[d]_{\delta^n_{X}}&Y\ar[d]^{\delta^n_{Y}}\\
\cF^{n}(X)\ar[r]^{\tilde L}& \cF^{n}(Y)
}\]
Furthermore,  $\|\tilde L \|_{\cb}=\Lip_n(L)$.  
\end{lemma}

\begin{proof}
Let $C:\Lip_0^{n}(Y,\C)\to\Lip_0^{n}(X,\C)$ be the map above, i.e.,   $C(f)=f\circ L$ for all $f\in \Lip_0^{n}(Y,\C)$. Clearly,  $\|C_k\|_{k}\leq \Lip_n(L)$ for all $k\in\N$. In order to get the reverse inequality, we proceed as in Proposition \ref{PropDeltaCompIsomIntoFreLipSp}. Let  $\eps>0$ and pick   $[x_{ij}],[y_{ij}]\in \M_n(X)$ such that \[\Lip_{n}(L)\leq\frac{ \|[L(x_{ij})-L(y_{ij})]\|_{\M_n(Y)}}{ \|[x_{ij}-y_{ij}]\|_{\M_n(X)}}+\eps.\] Fix a Hilbert space $H$ so that $ N\subset \cB(H)$ as an operator space and pick $\bar\xi=(\xi_i)_i,\bar\zeta=(\zeta_i)_i\in H^{\oplus n}$ such that 
\[|\langle[L(x_{ij})-L(y_{ij})]\bar \xi,\bar\zeta\rangle|
\geq \|[L(x_{ij})-L(y_{ij})]\|_{\M_n(Y)}-\eps\|[x_{ij}-y_{ij}]\|_{\M_n(X)}.\]
Let $K=\Span\{\xi_1,\ldots,\xi_n,\zeta_1,\ldots, \zeta_n\}$, let $P_K:H\to K$ be the orthogonal projection, and let $g:\cB(H)\to \cB(K)$ be the map $g(x)=P_Kx\restriction K$. Then $g$ is an $n$-contractive linear map. Without loss of generality, $\cB(K)=\M_k$ for $k=\dim(K)$. So $g$ is an element in the unit ball of $ \M_k(\cB(H)^*)$.

Let $\iota:\cB (H)^*\to \Lip_0^n(Y,\C)$ be the map given by Proposition \ref{PropMapIn}. So $\iota_k(g)$ is in the unit ball of $\M_k( \Lip_0^n(Y,\C))$ and this gives  that 
\begin{align*} \|C_k\|_k  & \geq \|C_k(\iota_k(g))\|_{\M_k}\\
&\geq \frac{\|g([L(x_{ij})-L(y_{ij})])\|_{\M_{nk}}}{\|[x_{ij}-y_{ij}]\|_{\M_n(X)}}\\
&\geq \frac{\|[L(x_{ij})-L(y_{ij})]\|_{\M_{n}(Y)}}{\|[x_{ij}-y_{ij}]\|_{\M_n(X)}}-\eps\\
&\geq \Lip_n(L)-2\eps.
\end{align*}
As $\eps$ was arbitrary, we conclude that $ \|C \|_{\cb}= \Lip_{n}(L)$, and it follows that $C^*: \Lip_0^{n}(X,\C)^*\to\Lip_0^{n}(Y,\C)^*$ also satisfies $\|C^*\|_{\cb}=\Lip_{n}(L)$.   

Given $a_1,\ldots, a_n\in \C$ and $x_1,\ldots, x_n\in X$, we have that \[C^*\Big(\sum_ia_i\delta^{n	}_{x_i}\Big)=\sum_ia_i\delta^{n}_{Lx_i}.\] Hence $\tilde L=C^*\restriction \cF^{n}(X)$ maps $\cF^{n}(X)$ into $\cF^{n}(Y)$ and it is clear that $\delta^n_{N}L=\tilde L\delta^n_{X }$. As $\|C^*\|_{\cb}=\Lip_{n}(L)$, we have that $\|\tilde L^*\|_{\cb}\leq \Lip_{n}(L)$. Moreover,   the proof that $\|C_k\|_k\geq \Lip_n(L)-2\eps$ above also implies that $\|\tilde L_k\|_k\geq \Lip_n(L)-2\eps$. Hence, as $\eps$ was arbitrary, we conclude that $\|\tilde L\|_{\cb}= \Lip_n(L)$.
\end{proof}

\begin{corollary}
Let $(X,x_0)$ and $(Y,y_0)$ be pointed operator metric spaces and let $n\in\N$. Let $L:X\to Y$ be an $n$-Lipschitz  equivalence with $L(x_0)=y_0$. Then $\tilde L$ is a complete isomorphism with complete bounded distorsion at most  $ \mathrm{Dist}_n(L)$, i.e., $\|\tilde L\|_{\cb}\|\tilde L^{-1}\|_{\cb}\leq \mathrm{Dist}_n(L)$. In particular,  if $L$ is a complete isometry,   $\tilde L$ is a complete linear isometry.
\end{corollary}
 
\begin{proof}
Since $L$ is surjective, then the map $\tilde L^{-1}:\cF^{n}(Y)\to \cF^{n}(X)$ given by Lemma \ref{LemmaLinearizationLipMaps} applied to $L^{-1}$ is clearly the inverse of $\tilde{L}$, so $\tilde L$ is surjective. Moreover, by Lemma \ref{LemmaLinearizationLipMaps}, $\|\tilde L\|_{\cb}\|\tilde L^{-1}\|_{\cb}= \Lip_{n}(L)\Lip_{n}(L^{-1})$. So, the result follows.
\end{proof}

Recall, given a metric space $X$ and $A\subset X$, a map $r:X\to A$ is called a \emph{retraction} if $f(a)=a$ for all $a\in A$. 

\begin{corollary}\label{CorSubspaceRetraction}
Let $(X,x_0)$ and $(Y,x_0)$ be  operator metric spaces with  distinguished points and so that $X\subset Y$. If there exists a  retraction $Y\to X$ with $n$-th Lipschitz constant at most $\lambda$, then $ \cF^{n}(X)$ is $\lambda$-completely isomorphic  to 
\[\overline{\Span\{\delta_x\in \cF^{n}(Y)\mid x\in X\}}^{\cF^{n}(Y)}.\]
\end{corollary}

 \begin{proof}
For each  $a_1,\ldots,a_n\in \C$ and $x_1,\ldots, x_n\in X$, define $\iota(\sum_ia_i\delta_{x_i} )=\sum_ia_i\delta_{x_i} $. Clearly, $\|\iota(\sum_ia_i\delta_{x_i} )\|_{\cF^{n}(Y)}\leq \|\sum_ia_i\delta_{x_i} \|_{\cF^{n}(X)}$, so $\iota$ extends to a contractive linear map $\iota:\cF^{n}(X)\to \cF^{n}(Y)$. Moreover, $\iota$ is clearly completely contractive.

Let $r:Y\to X$ be a   retraction with $\Lip_n(r)\leq \lambda$. Let $C:\Lip_0^{n}(X,\C)\to \Lip_0^{n}(Y,\C)$ be given by $C(f)=f\circ r$ for all $f\in \Lip^{n}_0(X,\C)$. By Lemma \ref{LemmaLinearizationLipMaps},  $\|C\|_{\cb}=\Lip_{n}(r)\leq \lambda$. Therefore, 
\[C_k\Big(B_{\Lip_0^{n}(X,M_k)}\Big)\subset \lambda\cdot B_{\Lip_0^{n}(Y,M_k)}\]
for all $k\in\N$.   By the definition of $\cF^{n}(Y)$, it follows that for all $k\in\N$ and all  $[z_{ij}]\in M_k(\cF^{n}(X))$, we have that $\|[\iota_k(z_{ij})]\|\geq \lambda^{-1}\|[z_{ij}]\|$. So, $\|\iota^{-1}_k\|\leq \lambda$ for all $k\in\N$. Hence $\iota$ is a $\lambda$-complete linear isomorphism into $\cF^{n}(Y)$.
\end{proof}

The map $\delta^n_{X}:X\to \cF^{n}(X)$ has a natural completely contractive left inverse if $X$ is an operator space. Indeed, let $X$ be an operator space. Given $a_1,\ldots, a_m\in \C$ and $x_1,\ldots, x_m\in X$, define 
\[\beta_X^n\Big(\sum_ia_i\delta_{x_i} \Big)=\sum_ia_ix_i.\] 
Given $x^*\in X^*$, we have that $x^*\in \Lip^{n}_0(X,\C)$, $\|x^*\|_{X^*}=\|x^*\|_{\Lip_0^{n}(X,\C)}$, and 
\[\Big|x^*\Big(\sum_ia_ix_i\Big)\Big|=\Big|\Big(\sum_ia_i\delta_{x_i} \Big)(x^*)\Big|\leq \Big\|\sum_ia_i\delta_{x_i} \Big\|\|x^*\|.\]
So, $\beta_X^n$ extends to a contractive map $\beta_X^n:\cF^{n}(X)\to X$. It is easy to check that $\beta_X^n$ is actually a complete contraction, hence a complete quotient map, and $\beta_X^n \delta^n_{X}=\mathrm{Id}_{X}$.

\begin{corollary}\label{CorLinearizationLipMapsTargetOS}
Let $(X,x_0)$  be  a pointed  operator metric space  and $Y$ be an operator space. Then, given  $L\in \Lip_0(X,Y)$, there exists a unique bounded linear map   $\bar L:\cF^{n}(X)\to Y$ such that $L=\bar L\delta^n_{X}$ and $\|\bar L\|_{\cb}= \|\bar L_n\|_{n}=\Lip_{n}(L)$.  
\end{corollary}

\begin{proof}
This is a trivial consequence of Lemma  \ref{LemmaLinearizationLipMaps} and the discussion preceding the lemma.
\end{proof}


Before stating the next corollary, recall that a linear operator $T : X \to Y$ between normed spaces
is a \emph{1-quotient}, or \emph{metric surjection}, if it is surjective and $\n{w} = \inf\big\{ \n{v} \mid Tv=w \big\}$ for every $w \in Y$.
A  \emph{complete 1-quotient},  or \emph{complete metric surjection}, is a map all of whose amplifications are 1-quotients/metric surjections.
On the other hand, given $C>0$, a Lipschitz map $f : X \to Z$ between metric spaces is called \emph{co-Lipschitz with constant $C$}
if for every $x \in X$ and $r>0$, we have  \[f\big( B(x,r) \big) \supseteq B(f(x),r/C).\]
A map is called a \emph{Lipschitz quotient} if it is surjective, Lipschitz, and co-Lipschitz \cite{BJLPS}. Similarly, we will say that a map between operator metric spaces is an \emph{$n$-Lipschitz quotient} if its $n$-th amplification is a Lipschitz quotient.

\begin{corollary}
Let $(X,x_0)$ and $(Y,y_0)$ be pointed operator metric spaces and let $n\in\N$. Suppose that $L\in \Lip_0(X,Y)$
is an $n$-Lipschitz quotient such that $L_n$ is Lipschitz with constant 1 and co-Lipschitz with constant 1.
Then its linearization $\tilde L:\cF^{n}(X)\to \cF^{n}(Y)$ is a complete 1-quotient.
\end{corollary}

\begin{proof} 
It suffices to show that $(\tilde{L})^* : \cF^{n}(Y)^* \to \cF^{n}(X)^*$ is a complete isometry.
Observe that this map can be canonically identified with the map
\[
f\in \Lip_0^{n}(Y,\C)\mapsto f\circ L\in \Lip_0^{n}(X,\C)
\]
in Lemma \ref{LemmaLinearizationLipMaps} and, by this lemma, we have $\| \tilde L \|_{\cb} = 1$.
Given $[f_{k\ell }]\in \M_m(\Lip_0^n (M,\C))$ and $\eps>0$, pick   $[x_{ij}],[y_{ij}]\in \M_n(Y)$ such that 
\[
\|[f_{k\ell }]\|_{\Lip,n,m} \|[x_{ij}-y_{ij}]\|_{\M_n(Y)} \leq (1+\eps)\|[f_{k\ell}(x_{ij})-f_{k\ell}(y_{ij})]\|_{\M_{nm}}.
\]
By the co-Lipschitz with constant 1 condition, we can find $[z_{ij}],[w_{ij}]\in \M_n(X)$ such that $L(z_{ij}) = x_{ij}, L(w_{ij}) = y_{ij}$ for $1 \le i,j, \le n$ and
\[
\n{[z_{ij} - w_{ij}]}_{\M_n(X)} \le  (1+\varepsilon)\n{[x_{ij} - y_{ij}]}_{\M_n(Y)}.
\]
It follows that
\[
\|[f_{k\ell } \circ L]\|_{\Lip,n,m} \ge (1+\eps)^{-2} \|[f_{k\ell } ]\|_{\Lip,n,m}.\qedhere
\]
\end{proof}

\subsection{Examples of $n$-Lipschitz-free operator spaces}\label{SubsectionExamples}

Given a connected undirected graph, we consider it (more precisely, its vertex set) as a metric space by endowing it with the shortest path distance. Recall that if such a graph has no cycles it is called a \emph{tree}, and it is well-known that the Lipschitz-free space of a tree on $k+1$ vertices can be isometrically identified with $\ell_1^k$ as Banach spaces \cite[Example 10.12]{OstrovskiiBook2013}.
The following proposition shows an example of an operator metric space where this identification is now at the level of operator spaces.

\begin{proposition}
Let $(T,x_0)$ be a rooted tree with $k+1$ vertices, say $T=\{0,1,\ldots, k\}$ where $x_0=0$. Consider the isometry $T \to \ell_1^k$ given  by
\[
0\in T \mapsto 0 \in \ell_1^k \  \text{ and } \ j \in T\setminus\{0\}\mapsto \sum_{0  \prec i \preccurlyeq j} e_i \in \ell_1^k,
\]
where $\prec$ is the partial order induced by the tree, 
and consider $T$ as an operator metric space with the structure induced   by this isometry and the maximal operator space structure on $\ell_1^k$.
Then for any $n\in\N$,  $\cF^n({T})$ is completely isometric to $\MAX(\ell_1^k)$.
\end{proposition}

\begin{proof}
Let $u : \cF^n({T}) \to \ell_1^k$ be the linear operator that for each $j\in \{1,2,\dotsc,k\}$ sends $\delta_j$ to $\sum_{0  \prec i \preccurlyeq j} e_i$. It is clear from Corollary \ref{CorLinearizationLipMapsTargetOS} that $u : \cF^n({T}) \to \MAX(\ell_1^k)$ is completely contractive, since it is the linearization of a complete isometry ${T} \to \MAX(\ell_1^k)$.

Since $\MAX(\ell_1^k)$ is a maximal space,
\[
\big\| u^{-1} : \MAX(\ell_1^k) \to \cF^n({T}) \big\|_{\cb} = \big\| u^{-1} : \ell_1^k \to \cF^n({T}) \big\|,
\]
and now we estimate
\[
\big\| u^{-1} : \ell_1^k \to \cF^n({T}) \big\| \le \big\| u^{-1} : \ell_1^k \to \cF^1({T}) \big\| \cdot \big\| \mathrm{Id} : \cF^1({T}) \to \cF^n({T}) \big\|. 
\]
Note that $\big\| u^{-1} : \ell_1^k \to \cF^1({T}) \big\| = 1$ follows from the classical identification of $\cF^1({T})$ with $\ell_1^k$. Also, as $\Lip(f)\leq \Lip_n(f)$ (see \eqref{Eq.LipLipnEquiv}), we have that the idendity $\Lip_0^n(T)\to \Lip_0(T)$ has norm at most $1$, which in turn implies that  $\big\| \mathrm{Id} : \cF^1({T}) \to \cF^n({T}) \big\| \le 1$ (alternatively, this will   also be a consequence of   Theorem \ref{Thm-Alternative-N-Lip-Free} below).
\end{proof}

The previous proposition is a particular case of the following more general situation, whose proof is exactly the same.

\begin{proposition}\label{PropMaximalnLipSpace}
Let $(X,x_0)$ be a pointed metric space. Consider the maximal operator space structure $\MAX(\cF(X))$ on $\cF(X)$,
and induce an operator metric space structure on $X$ by the canonical embedding $\delta_X : X \to \cF(X)$. 
Then for each  $n\in\N$, the $n$-Lipschitz free space $\cF^n({X})$ is completely isometric to $\MAX(\cF(X))$. 
\end{proposition}



\begin{remark}
Suppose that $(X,x_0)$ is a pointed metric space corresponding to a finite graph endowed with the shortest path distance.
It is known \cite[Prop. 10.10]{OstrovskiiBook2013} that in this case $\cF(X)$ can be identified (as a Banach space) with $\ell_1(\mathcal{E})/Z$, where $\mathcal{E}$ is the edge set of the graph and $Z$ is the \emph{cycle subspace} (see \cite{OstrovskiiBook2013} for the detailed definition, and note that though the proof of \cite[Prop. 10.10]{OstrovskiiBook2013} is given for real scalars the same argument works in the complex case).
Since quotients of maximal spaces are maximal, endowing $X$ with an operator metric space structure as in Proposition \ref{PropMaximalnLipSpace} yields that for all $n\in\N$ the $n$-Lipschitz free space $\cF^n(X)$ is completely isometric to $\MAX(\ell_1(\mathcal{E}))/Z$.
\end{remark}


\subsection{$n$-Lipschitz-free $\R$-operator spaces}\label{SubsectionRealLipFree}
In Banach space theory, $\cF(X)$ is usually defined with respect to real-valued Lipschitz maps $X\to \R$ with $f(x_0)=0$; in particular, it is an $\R$-Banach space. We discuss this approach in this subsection (we refer the reader to \cite{Ruan-real-OS}  for details on $\R$-operator spaces).

 Proceeding analogously as we did above for $\Lip_0(X,\C)$,  we denote  by $\Lip_0^n(X,\R)$ the $\R$-operator space which consists of $\Lip_0(X,\R)$ endowed with the real operator space structure  $(\|\cdot\|_{\Lip,n,k})_k$, where the sequence of norms $(\|\cdot\|_{\Lip,n,k})_k$ is defined completely analogously as above. Given $x\in X$, in an abuse of notation we denote by $\delta_x$ the map $\Lip_0(X,\R)\to \R$ given by $\delta_x(f)=f(x)$ for all $f\in \Lip_0(X,\R)$. The \emph{ $n$-Lipschitz-free $\R$-operator space of $X$} is the $\R$-Banach space 
\[\cF^n_{\R}(X)=\overline{\Span_\R\{\delta_X\in \Lip_0^n(X,\R)^*\mid x\in X \}}\]
together with the operator space structure inherited from $\Lip_0^n(X,\R)^*$. If $n=1$, we write $\cF_\R(X)=\cF^1_\R(X)$.  

All the results in Section \ref{SectionLipFreeSp} have analogous versions for the  $n$-Lipschitz-free $\R$-operator spaces and their proofs follow   analogously as well (this will be used in Section \ref{SectionLiftProp}).

When working in the $\R$-Banach space category, an advantage of working with $\cF_\R(X)$  is that a real-valued Lipschitz map can always be extended without increasing its Lipschitz constant (see \cite[Section 2]{Godefroy2015Survey}). For complex-valued Lipschitz maps this is no longer the case (\cite[Example 1.37]{Weaver1999BookSecondEdition}) --- the Lipschitz constant may increase by a factor of at most $\sqrt{2}$ (\cite[Corollary 1.34]{Weaver1999BookSecondEdition}). As a result of that, if $Y\subset X$, then $\cF(Y)$ is $\sqrt{2}$-isomorphic to a subspace of $\cF(X)$, while $\cF_{\R}(Y)$ is linearly isometric to a subspace of $\cF_{\R}(X)$. However, in the operator space category, it is not clear if this advantage remains. Precisely:

\begin{problem}
Let $X$ and $Y$ be operator metric spaces with $X\subset Y$ and let $n\in\N$. Let $f:X\to \R$ be a Lipschitz map. Is there an extension $F:Y\to \R$ of $f$ so that $\|F\|_{\Lip, n}\leq \|f\|_{\Lip,n}$?
\end{problem}

  At last, we finish this subsection relating  the $n$-Lipschitz-free $\R$-operator space with its complex version defined before. Its proof is completely   straightforward, so we leave the details to the reader. 

\begin{proposition}\label{PropRealVSComplexLipFreeSp}
Let $X$ be a $\C$-operator space. Then $\cF^n(X)$ is $\R$-isomorphic to $\cF_\R^n(X)\oplus \cF_\R^n(X)$ for all $n\in\N$.\qed 
\end{proposition}

\section{$n$-Lipschitz-lifting property}\label{SectionLiftProp}

In \cite{GodefroyKalton2003}, the authors introduced  the isometric Lipschitz-lifting property for Banach spaces, and showed that every separable Banach space satisfies this property. In this subsection, we introduce the equivalent definition and prove the equivalent statement for operator spaces. This will allow us to obtain Theorem \ref{ThmCompIsomEmbImpliesCompLinIsomEmb}.

\begin{definition}(cf.~\cite[Definition 2.7]{GodefroyKalton2003})
Let $X$ be an operator space and let $n\in\N$. We say that $X$ has the \emph{$n$-isometric Lipschitz-lifting property} if there exists
a linear $n$-contraction $T:X\to \cF^{n}(X)$ such that $\beta_X^n T=\mathrm{Id}_{X}$.\label{DefiLipLiftProp}
\end{definition}

Notice that we always have  $\beta_X^n\circ \delta_X^n=\mathrm{Id}_X$. So the content of the definition above lies on $T$ being linear. In order to show that   separable operator spaces have the completely isometric Lipschitz-lifting property, we will need to work with Gateaux differentiability.

\begin{definition}(cf.~\cite[Definition 14.2.1]{AlbiacKaltonBook})\label{DefinitionGateauxDiff}
Let $X$ and $Y$ be  $\R$-Banach spaces.
 A map   $f:X\to Y$  is \emph{ Gateaux $\R$-differentiable at $x\in X$} if for all $a\in X$ the limit 
\[Df_x(a)=\lim_{\lambda\to 0}\frac{f(x+\lambda a)-f(x)}{\lambda}\]
exists and the map $a\in X\mapsto Df_x(a)\in Y$ is $\R$-linear and   bounded.
If $f:X\to Y$ is Gateaux $\R$-differentiable at every $x\in X$, we simply say that $f$ is \emph{ Gateaux $\R$-differentiable}.
If $X$ and $Y$ are $\C$-Banach spaces, we say that $f:X\to Y$ is  \emph{Gateaux $\R$-differentiable at $x\in X$} if it is so with $X$ and $Y$ being seen as  $\R$-Banach spaces.
\end{definition}

Since operator spaces are, in particular, Banach spaces, the notion of Gateaux $\R$-differentiability applies to operator spaces. The following simple proposition shows how  operator norms of   derivatives relate to their Lipschitz operator norms. 

 \begin{proposition}\label{PropComplRadeThm1}
Let $X$  and  $Y$ be    operator spaces, $n,k\in\N$, and $[f_{\ell m}]\in \M_k(\Lip_0^n(X,\C))$ be   Gateaux $\R$-differentiable at $x\in X$.  Then   \[\|D([f_{\ell m}])_x\|_n\leq \|[f_{\ell m}]\|_{\Lip,n,k}.\]
\end{proposition}

\begin{proof}
Given $[a_{ij}]\in \M_n(X)$,    we have
\begin{align*}
\|D([f_{\ell m}])_x([a_{ij}])\|_{\M_n(Y)} &= \lim_{\lambda\to 0}\Big\|\Big[  \frac{f_{\ell m }(x+\lambda a_{ij})-f_{\ell m }(x)}{\lambda}\Big]\Big\| _{n k}\\
&= \lim_{\lambda\to 0}\frac{1}{\lambda}\cdot \Big\|  [f_{\ell  m}]_n([x+\lambda a_{ij}])-[f_{\ell m}]_n([x])\Big\|_{n k}\\
&\leq \Lip_n([f_{\ell m}]) \|[a_{ij}]\|_{\M_n(X)}
\end{align*}
 So,   $\|D([f_{\ell m}])_x\|_n\leq\Lip_n([f_{\ell m}])= \|[f_{\ell m}]\|_{\Lip,n,k}$.
\end{proof}

Before we prove that separable operator spaces have the completely isometric Lipschitz-lifting property, we need a proposition about the density of the set of Gateaux differentiable functions in $\Lip_0^{n}(X,\C)$. The next proposition is the operator space version of \cite[Corollary 6.43]{BenyaminiLindenstraussBook}.

\begin{proposition}\label{PropDensityGateauxDiff}
Let $X $ be a separable operator space. Then, for all $n,k\in\N$, all $[f_{\ell m}]\in \M_k(\Lip_0 ^n(X,\C))$,  and all $\eps>0$, there is a Gateaux $\R$-differentiable  $[g_{\ell m}]\in \M_k(\Lip_0^n (X,\C))$ such that $\|[g_{\ell m}]\|_{\Lip,n,k}\leq \|[f_{\ell m}]\|_{\Lip,n,k}$ and 
\[\sup_{[x_{ij}]\in \M_n(X)}\|[f_{\ell m}(x_{ij})-g_{\ell m}(x_{ij})]\|_{\M_{nk}}\leq \eps.\] 
\end{proposition}

\begin{proof}
Since $X$ is a separable Banach space, there exists a  nondegenerate Gaussian measure $\mu$ on $X$ so that $\int_{X}\|z\|d\mu(z)\leq 1$.\footnote{See \cite[Definition 6.17]{BenyaminiLindenstraussBook} for the definition of a Gaussian measure on a Banach space and the comments after \cite[Proposition 6.20]{BenyaminiLindenstraussBook} for a proof of this statement.} Fix $n,k\in\N$ and $[f_{\ell m}]\in \M_k(\Lip_0^{n}(X,\C))$. For each $\ell,m\in \{1,\ldots, k\}$ and $N\in\N$, define $f^N_{\ell m}:X\to \C$ by letting \[f^N_{\ell m}(x)=\int_{X}f_{\ell m}(x+N^{-1}z)d\mu(z).\] Given  $[x_{ij}],[y_{ij}]\in \M_n(X)$, we have that 
\begin{align*}
\|[f^N_{\ell m}(x_{ij})-&f^N_{\ell m}(y_{ij})]\|_{\M_{nk}}\\
&=\Big\|\int_{X}[f_{\ell m}(x_{ij}+N^{-1}z)-f_{\ell m}(y_{ij}+N^{-1}z)]d\mu(z)\Big\|_{\M_{nk}}\\
&\leq \int_{X} \|[f_{\ell m}(x_{ij}+N^{-1}z)-f_{\ell m}(y_{ij}+N^{-1}z)]\|_{\M_{nk}}d\mu(z)\\
&\leq \|[f_{\ell m}]\|_{\Lip,n,k}\|[x_{ij}]-[y_{ij}]\|_{\M_{n}(X)}.
\end{align*}
So, $\|[f^N_{\ell m}]\|_{\Lip,n,k}\leq\|[f_{\ell m}]\|_{\Lip,n,k}$ for all $N\in\N$.  Moreover, by the Lebesgue dominated convergence theorem and \cite[Theorem 6.42]{BenyaminiLindenstraussBook}, each $f^N_{\ell m}$ is everywhere Gateaux $\R$-differentiable for all $\ell ,m\in \{1,\ldots, k\}$. 

Fix $n\in\N$, and notice that, for all $[x_{ij}]\in \M_n(X)$, we have
\begin{align*}
\|[f_{\ell m}(x_{ij})-f^N_{\ell m}(x_{ij})]\|_{\M_{nk}}&\leq \int_{E}\|[f_{\ell m}(x_{ij})-f_{\ell m}(x_{ij}+N^{-1}z)]\|_{\M_{nk}}d\mu(z)\\
&\leq N^{-1}\|[f_{\ell m}]\|_{\Lip,n,k}\int_E\|[z]\|_{\M_n(X)}d\mu(z)\\
&\leq n^2 N^{-1}\|[f_{\ell m}]\|_{\Lip,n,k}
\end{align*}
Hence, given any $\eps>0$, there exists $N\in \N$ large enough so that \[\|[f_{\ell m}(x_{ij})-f^N_{\ell m}(x_{ij})]\|_{\M_{nk}}<\eps\] for all $[x_{ij}]\in \M_n(X)$.
\end{proof}

\begin{theorem}\label{ThmIsoLipLifProp}
Every separable operator space has the $n$-isometric Lipschitz-lifting property for all $n\in\N$.
\end{theorem}

\begin{proof}
Let $X$ be a separable  operator space, $n\in\N$,  and let $(x_m)_m$ be a linearly independent sequence in $X$ whose linear span $X_0$ is dense in $X$ and so that \[\Big\{\sum_mt_mx_m\mid (t_m)_m\in [0,1]^\N\Big\}\] is a compact subset of $X$. Define $L:[0,1]^\N\to X$ by letting $L(\bar t)=\sum_mt_mx_m$ for all $\bar t=(t_m)_m\in [0,1]^\N$, and let $\lambda$ denote the product measure of the Lebesgue measure on $[0,1]$.

Proceeding exactly as in   the proof of  \cite[Theorem 3.1]{GodefroyKalton2003}, there exists a linear map $R:X_0\to \cF^{n}(X)$ so that $\beta_X^n R(x_m)=x_m$ for all $m\in\N$ and 
\[R(x)(f)=\int_{[0,1]^ \N} Df_{L(\bar t)}(x)  d\lambda(\bar t)\]
for all $x\in X_0$ and all Gauteaux $\R$-differentiable maps $f\in \Lip_0^{n}(X,\C)$. (We leave the details to the reader.) Hence, given $k\in\N$, $[x_{ij}]\in \M_n(X_0)$, and a Gateaux $\R$-differentiable  $[f_{\ell m}]\in \M_k(\Lip_0^{n}(X,\C))$, we have that
\begin{align*}
\|[R(x_{ij})(f_{\ell m})]\|_{\M_{nk}}&=\Big\|\Big[\int_{[0,1]^ \N} D (f_{\ell m})_{L(\bar t)}(x_{ij})  d\lambda(\bar t)\Big]\Big\|_{\M_{nk}}\\
&\leq \int_{[0,1]^ \N}\|[ D (f_{\ell m})_{L(\bar t)}(x_{ij}) ]\|_{\M_{nk}} d\lambda(\bar t)\\
&\leq \int_{[0,1]^ \N}\| [D(f_{\ell m})_{L(\bar t)} ]\|_{n}\|[x_{ij}]\|_{\M_n(X)}d\lambda(\bar t).
\end{align*}
By Proposition \ref{PropComplRadeThm1}, we have that  $\| [D (f_{\ell m})_{y}]\|_{n}\leq  \|[f_{\ell m}]\|_{\Lip,n,k}$, so  \[\|[R(x_{ij})(f_{\ell m})]\|_{\M_{nk}}\leq \|[f_{\ell m}]\|_{\Lip,n,k}\|[x_{ij}]\|_{\M_n(X)}.\] By Proposition \ref{PropDensityGateauxDiff}, the subset of the unit ball of $\M_k(\Lip_0^n(X,\C))$ consisting of all Gateaux $\R$-differentiable maps is  dense with respect to the uniform convergence topology in this unit ball. Therefore, by Proposition \ref{prop-duality-for-F^n}, this subset is also weak$^*$ dense in it. This shows that  
\[\|[R(x_{ij})(f_{\ell m})]\|_{\M_{nk}}\leq \|[f_{\ell m}]\|_{\Lip,n,k}\|[x_{ij}]\|_{\M_n(X)} \]
for all  $[f_{\ell m}]\in \M_k(\Lip_0^{n}(X,\C))$, which in turns implies that 
$R$ extends to a complete $n$-contraction  $T:X\to \cF^{n}(X)$. Since $\beta_X^n R(x_n)=x_n$ for all $n\in\N$, it follows that $\beta_X^n T=\mathrm{Id}_{X}$, and we are done.
\end{proof}

The next proposition is the operator space version of a result of T. Figiel \cite{Figiel1968} (see also \cite[Theorem 14.4.10]{AlbiacKaltonBook}).

\begin{proposition}\label{PropFigielComplete}
Let $X$ and $Y$ be  operator spaces, $n\in\N$, and let $f:X\to Y$ be an $n$-isometry so that $\overline{\Span}_\R\{f(X)\}=Y$ and $f(0)=0$. Then there exists a unique $\R$-linear map $T:Y\to X$ so that $\|T_n\|_{n}=1$ and $T\circ f=\mathrm{Id}_{X}$.
\end{proposition}

\begin{proof}
  Since  $\overline{\Span}_\R\{f(X)\}=Y$, it follows that \[\overline{\Span}_\R\{f_k(X)\}=\M_k(Y)\]
  for all $k\in\N$. Hence, as $f_k$ is an isometry for $k\leq n$, the result for Banach spaces (see \cite[Theorem 14.4.10]{AlbiacKaltonBook}, or \cite{Figiel1968}) applied to the map $f_k:\M_k(X)\to \M_k(Y)$ gives us an $\R$-linear map $T^k:\M_k(Y)\to \M_k(X)$ so that $T^k\circ f_k=\mathrm{Id}_{\M_k(X)}$ and $\|T^k\|=1$. Moreover,   this is the unique $\R$-linear  map with such properties.  If $k=1$, we simply write $T=T^1$; so $T\circ f=\mathrm{Id}_{X}$.

Let us notice that
$T^n=T_n$, i.e., $T^n$ is the $n$-th amplification of $T$.  For that, fix $i,j\in \{1,\ldots, n\}$, let $I_{ij,X}:X\to \M_n(X)$ and $I_{ij,Y}:Y\to \M_n(Y)$ be the natural inclusions into the $(i,j)$-th coordinate of $\M_n(X)$ and $\M_n(Y)$, respectively, and  let $\pi_{ij,X}:\M_n(X)\to X$ be the projection onto the $(i,j)$-th coordinate. Define $T^n_{ij}:Y\to X$ by   \[T^n_{ij}=\pi_{ij,X}\circ T^n\circ I_{ij,Y}.\] As $f(0)=0$, we have $I_{ij,Y}\circ f=f_n\circ I_{ij,X}$, and it follows that 
\[T^n_{ij}\circ f=\pi_{ij,X}\circ T^n\circ f_n\circ I_{ij,X}=\pi_{ij,X}\circ\mathrm{Id}_{\M_n(X)}\circ I_{ij,X}=\mathrm{Id}_{X}.\]
 As $\|T^n\|=1$, it is clear that $\|T^n_{ij}\|= 1$. Hence, as  $T=T^1$ is the unique $\R$-linear operator $T:Y\to X$ so that $\|T\|=1$ and $T\circ f=\mathrm{Id}_{X}$,  the equality above implies that $T^n_{ij}=T$. Since $i$ and $j$ are arbitrary,  $T^n$ is the $n$-th amplification of $T$.

Since $\|T^n\|=1$, it follows that  $\|T_n\|_{n}=1$, and we are done.
\end{proof}

The following should be compared with \cite[Proposition 2.9]{GodefroyKalton2003}.

\begin{proposition}\label{PropLiftIsoLiftPropSections}
Let $X$ and $Y$ be   operator spaces, $n\in\N$,  and assume that $X$ has the $n$-isometric Lipschitz-lifting property. Let $Q:Y\to X$ be an $\R$-linear $n$-contractive  surjection. If $Q$ admits an  $n$-isometric section, then $Q$ admits an $\R$-linear $n$-isometric section. 
\end{proposition}

\begin{proof}
Let $L:X\to Y$ be an $n$-isometric section of $Q$ and let $\bar L:\cF^n_\R(X)\to Y$ be given by Corollary  \ref{CorLinearizationLipMapsTargetOS} (see Subsection \ref{SubsectionRealLipFree}),  so $\bar L$ is an $\R$-linear $n$-isometry so that $L=\bar L\delta^n_X$. Since $L$ is a section of $Q$, it follows that $Q\bar L\delta^n_X=\mathrm{Id}_X$; so, we must have  $Q\bar L =\beta_X^n$. As $X$ has the $n$-isometric Lipschitz-lifting property, let $T:X\to \cF^n_\R(X)$ be the $\R$-linear $n$-contraction with $\beta_X^n T=\mathrm{Id}_X$. We must then have that $Q\bar L T =\mathrm{Id}_X$, so $\bar L T$ is an $\R$-linear $n$-isometric section of $Q$.
\end{proof}

\begin{proof}[Proof of Theorem \ref{ThmCompIsomEmbImpliesCompLinIsomEmb}]
Let $(f^n)_n$ be an almost completely isometric embedding of $X$ into $Y$. Fix $n\in\N$  and set $Z=\overline{\Span}_\R\{f^n(X)\}$. By Proposition \ref{PropFigielComplete}, there exists an $\R$-linear $n$-contraction $T:Z\to X$ so that $T\circ f^n=\mathrm{Id}_{X}$, i.e., $f^n$ is an $n$-isometric   section of $T$. Since $X$ is separable it has the $n$-isometric Lipschitz-lifting property by Theorem \ref{ThmIsoLipLifProp}, and  Proposition \ref{PropLiftIsoLiftPropSections} implies that $T$ admits an $\R$-linearly $n$-isometric section $u:X\to Z$. As $n\in\N$ was arbitrary, we are done.
\end{proof}

We finish this section by pointing out that there are nonseparable examples of operator spaces having the $n$-isometric Lipschitz-lifting property: simply take a nonseparable pointed operator metric space in the next proposition. 
The proof is exactly the same as that of \cite[Lemma 2.10]{GodefroyKalton2003}, so we omit it.
 
\begin{proposition}
Let $(X,x_0)$ be a pointed operator metric space.
For any $n\in\N$, the $n$-Lipschitz-free space $\cF^n(X)$ has the $n$-isometric Lipschitz-lifting property. 
\end{proposition}

\section{An alternative approach to $n$-Lipschitz-free operator spaces}\label{SectionAlternative}

In Definition \ref{DefNLipFree}, we have defined the operator space structure of $\cF^n(X)$ in a ``dual way''. Indeed, this is done by embedding $\cF^n(X)$ into the operator space $\Lip^n_0(X,\C)^*$.
Therefore, the norm of an element of $\M_m(\cF^n(X))$ is naturally calculated as a supremum.
It is always useful to have an alternative description of such a quantity as an infimum, which is what we will do in this section.
Our approach in this section mirrors the presentation of the classical case as in \cite{Weaver1999BookSecondEdition}, which in turn follows that of \cite{Arens-Eells}. 
 
Recall from Remark \ref{RemNLipFreeIsNMaximal} above that $\cF^n(X)$ is an $n$-maximal operator space, so
Proposition \ref{PropMaxAsInf} below goes in the desired direction: for an $m$-maximal operator space, it gives a description of the operator space structure as an infimum.
It is a generalization of \cite[Theorem 3.1]{Pisier-OS-book}.
The proof is exactly the same, just using the description of $\MAX_m(E)$ as in \cite[Section 2]{OikhbergRicard2004MathAnn} or \cite[Proposition I.3.1]{LehnerPhDThesis}.
Proposition \ref{PropMaxAsInf} is not needed for what we are doing, but it provided the inspiration for Theorem \ref{Thm-Alternative-N-Lip-Free} below (and it might be of independent interest).

Given an operator space $E$ and $n,k\in\N$, we view operators $a\in \M_k(\M_n(E))$ as $k$-by-$k$ matrices of $m$-by-$m$ matrices with entries in $E$.

\begin{proposition}\label{PropMaxAsInf}
Let $E$ be an operator space and $n\in\N$, $m\in\N$. For each $x \in \M_n(E)$ we have  
\[
\n{x}_{\M_n(\MAX_m(E))} = \inf \big\{ \n{\alpha} \n{D} \n{\beta}  \mid x = \alpha \cdot D \cdot \beta \big\},
\]
where the infimum above is taken over all $N\in\N$, all $\alpha \in \M_{n,Nm}$, all $\beta \in \M_{Nm,n}$, and  all $N\times N$ diagonal matrices $D \in \M_{N}(\M_m(E))$ with entries in $\M_m(E)$.
\end{proposition}

\begin{proof}
For each $n$ define a norm $\tn{\cdot}_n$ on $\M_n(E)$ as the infimum appearing in the statement.
It is not difficult to see that this sequence of norms satisfies Ruan's axioms as in \cite[Section 2.2]{Pisier-OS-book}, and  therefore it defines an operator space structure on $E$ that we denote by $\tilde{E}$.

Let $F$ be an arbitrary operator space and $u: E\to F$ be linear and $m$-contractive. It follows easily from Ruan's axioms that $u: \tilde E\to F$ is completely contractive; hence, $\tilde E$ is completely isometric to $\MAX_m(E)$ by \cite[Lemma 2.3]{OikhbergRicard2004MathAnn}. \qedhere
\end{proof}

We can now give the promised alternative description of $\cF^n(X)$. It shows that in the specific case of an $n$-Lipschitz-free operator space, the representations appearing in Proposition \ref{PropMaxAsInf} can be assumed to be of a special form.

\begin{theorem}\label{Thm-Alternative-N-Lip-Free}
Let $(X,x_0)$ be a pointed operator metric space, $n\in\N$, and let $F = \Span\{\delta_x\}_{x\in X}\subset \cF^n(X)$. 
For any $m\in\N$ and $\mu \in \M_m(F)$ we have
\[
\n{\mu}_{\M_m(\cF^n(X))} = \inf \Big\{ \n{\alpha} \n{\beta} \max_{1 \le \ell \le N} |c_\ell| \n{[x^\ell_{ij} - y^\ell_{ij}]} \Big\}
\]
where the infimum is taken over all $N\in\N$ and all representations of $\mu$ of the form $\mu = \alpha \cdot D \cdot \beta$ where $\alpha \in \M_{m,Nn}$ and $\beta \in \M_{Nn,m}$ are scalar matrices, and $D \in \M_{N}(\M_n(F))$  is a diagonal matrix whose diagonal entries are of the form $c_\ell[\delta_{x^\ell_{ij}} - \delta_{y^\ell_{ij}}]_{ij}$ with $c_\ell$ a scalar and $[x^\ell_{ij}],[y^\ell_{ij}] \in \M_n(X)$ for  $1 \le \ell \le N$.
\end{theorem}

\begin{proof}
Use the infimum in the statement to define a norm $\tn{\cdot}_m$ on each space $\M_m(F)$.
Note that just as in the classical case, in principle it is only clear that this defines seminorms. But since $\tn{\cdot}_1$ coincides with the Arens-Eells definition of the norm on the Lipschitz-free space $\cF(X)$ \cite[Section  3.1]{Weaver1999BookSecondEdition}, it follows from the discussion after \cite[Definition 3.2]{Weaver1999BookSecondEdition} that $\tn{\cdot}_1$ is a norm and thus so is $\tn{\cdot}_m$ for each $m$.
It is not difficult to see that this sequence of norms satisfies Ruan's axioms (\cite[Section 2.2]{Pisier-OS-book}), and therefore defines an operator space structure on the completion of $F$ with respect to $\tn{\cdot}_1$. To set notation, denote this operator space by $\fbs^n(X)$.
We now show that $\fbs^n(X)^*$ is completely isometric to $\Lip_0^{n}(X,\C)$ using the same maps as in Proposition \ref{prop-duality-for-F^n}. This will  imply  that $\fbs^n(X)$ and $\cF^n(X)$ are completely isometric via the identity map.

As before (cf.  Proposition \ref{prop-duality-for-F^n}), we define a map $u:\Lip_0^{n}(X,\C)\to \fbs^{n}(X)^*$ by letting \[u(f)\Big(\sum_ia_i\delta_{x_i}\Big)=\sum_ia_if(x_i)\] for all   $a_1,\ldots, a_m\in\C$, all $x_1,\ldots, x_m\in X$, and all $f\in \Lip_0^{n}(X,\C)$. Clearly $u(f)$ is linear on $F$, and the calculations below show that $u(f)$ extends to a linear functional on all of $\fbs^{n}(X)$.

Fix $[f_{rs}] \in \M_k(\Lip_0^{n}(X,\C))$.
Take $\mu = [\mu_{ab}] \in \M_m(F)$, and consider a representation $\mu = \alpha \cdot D \cdot \beta$ as in the statement of the theorem.
Now,
\begin{multline*}
\n{ [u(f_{rs})\mu_{ab}] }_{\M_{km}} = 
\n{ \big[ \alpha \cdot \diag\big( c_\ell[f_{rs}(x^\ell_{ij}) - f_{rs}(y^\ell_{ij})]_{ij} \big) \cdot \beta \big]_{rs} }_{\M_k(\M_m)} \\
= \n{ (I_k \otimes \alpha) \cdot \big[ \diag\big( c_\ell[f_{rs}(x^\ell_{ij}) - f_{rs}(y^\ell_{ij})]_{ij} \big) \big]_{rs} \cdot (I_k \otimes \beta) }_{\M_k(\M_m)} \\
\le \n{I_k \otimes \alpha} \n{ \big[ \diag\big( c_\ell[f_{rs}(x^\ell_{ij}) - f_{rs}(y^\ell_{ij})]_{ij} \big) \big]_{rs} }_{\M_k(\M_N(\M_n))} \n{I_k \otimes \beta}\\ 
= \n{\alpha} \n{\beta} \n{  \diag\big( c_\ell[f_{rs}(x^\ell_{ij}) - f_{rs}(y^\ell_{ij})]_{ijrs} \big) }_{\M_N(\M_{kn}))} \\
=
 \n{\alpha} \n{\beta} \max_{1 \le \ell \le N} |c_\ell| \n{ \big[f_{rs}(x^\ell_{ij}) - f_{rs}(y^\ell_{ij})\big]_{ijrs}} \\
 \le \n{\alpha} \n{\beta} \n{[f_{rs}]}_{\M_k(\Lip_0^{n}(M,\C))} \max_{1 \le \ell \le N} |c_\ell| \n{[x^\ell_{ij} - y^\ell_{ij}]}.
\end{multline*}
By taking the infimum over all representations of $\mu$ we conclude that
\[
\n{ [u(f_{rs})\mu_{ab}] }_{\M_{km}} \le \n{[f_{rs}]}_{\M_k(\Lip_0^{n}(M,\C))} \tn{\mu}_m,
\]
which shows that $u:\Lip_0^{n}(X,\C)\to \fbs^{n}(X)^*$ is a complete contraction.

As before (cf. Proposition \ref{prop-duality-for-F^n}), let $v:\fbs^{n}(X)^*\to \Lip_0^{n}(X,\C)$ be given by $v(g)(x)=g(\delta_x)$ for all $g\in \fbs^{n}(X)^*$ and all $ x\in X$; the calculations below will show that this is well defined.

Let $[g_{rs}] \in \M_m(\fbs^{n}(X)^*)$, and consider $[x_{ij}],[y_{ij}] \in \M_n(X)$.
Now,
\begin{multline*}
\n{ [v(g_{rs})(x_{ij}) - v(g_{rs})(y_{ij})] }_{\M_{mn}} = \n{ [g_{rs}(\delta_{x_{ij}} - \delta_{y_{ij}})] }_{\M_{mn}} \\
\le \n{[g_{rs}]}_{\M_m(\fbs^{n}(X)^*)} \n{ [\delta_{x_{ij}} - \delta_{y_{ij}}] }_{\M_n(\fbs^{n}(X))} \\
\le \n{[g_{rs}]}_{\M_m(\fbs^{n}(X)^*)} \n{ [x_{ij} - y_{ij}] }_{\M_n(X)},
\end{multline*}
where in the last inequality we have used the representation $[\delta_{x_{ij}} - \delta_{y_{ij}}] = I_n \cdot [\delta_{x_{ij}} - \delta_{y_{ij}}] \cdot I_n$ to estimate the norm $\n{ [\delta_{x_{ij}} - \delta_{y_{ij}}] }_{\M_n(\fbs^{n}(M))}$. The inequality above shows that $v:\fbs^{n}(X)^*\to \Lip_0^{n}(X,\C)$ is a complete contraction, finishing the proof.
\end{proof}

While Theorem \ref{Thm-Alternative-N-Lip-Free} provides a description for the norms of matrices over $\cF^n(X)$ whose entries come from $\Span\{\delta_x\}_{x\in X}$, it is also desirable to have a description that applies to all matrices over $\cF^n(X)$.
That is the content of our next result, whose proof is based on that of the usual representation for elements of the completion of the projective tensor product of operator spaces (\cite[Theorem 10.2.1]{Effros-Ruan-book}).
This is not a surprise: in the classical case, it is well-known that one can represent elements of the Lipschitz-free space  using a series representation that is reminiscent of what one does for projective tensor products of Banach spaces: e.g., compare \cite[Prop. 2.8]{Ryan} and 
\cite[Lemma 2.1]{AliagaPernecka2020}.

The notation for infinite matrices which we use below is that of \cite{Effros-Ruan-book}. The reader can find all details in Sections 1.1, 10.1 and 10.2 of \cite{Effros-Ruan-book}, we will only recall some basics.
Given an operator space $E$, $\M_\infty(E)$ is the space of infinite matrices $[x_{ij}]_{i,j=1}^\infty$ with entries $x_{ij}$ in $E$ whose truncations to $\M_n(E)$ are uniformly bounded; $\M_\infty(E)$ has a natural operator space structure.
A similar construction can be done for matrices $\M_{I,J}(E)$, where $I$ (resp. $J$) is a finite or countably infinite set of indices for the rows (resp. columns); operator space structures and matrix multiplication can then be defined using truncations again (in this context,   $\N$ is identified with $\infty$). Moreover, if $E=\C$, we simply write $\M_\infty=\M_\infty(\C)$, $\M_{n,\infty}=\M_{n,\infty}(\C)$, and $\M_{\infty,n}=\M_{\infty,n}(\C)$.

\begin{theorem}\label{Thm-Alternative-N-Lip-Free-Completion}
Let $(X,x_0)$ be a pointed operator metric space  and $n\in\N$.
For any $m\in\N$ and $\mu \in \M_m(\cF^n(X))$ we have
\[
\n{\mu}_{\M_m(\cF^n(X))} = \inf \Big\{ \n{\alpha} \n{\beta} \sup_{\ell} |c_\ell| \n{[x^\ell_{ij} - y^\ell_{ij}]} \Big\}
\]
where the infimum is taken over all representations of $\mu$ of the form $\mu = \alpha \cdot D \cdot \beta$ where $\alpha \in \M_{m,\infty }$ and $\beta \in \M_{\infty ,m}$ are scalar matrices, and $D \in \M_{\infty}(\M_n(\Span\{\delta_x\}_{x\in X}))$  is a diagonal matrix whose diagonal entries are of the form $c_\ell[\delta_{x^\ell_{ij}} - \delta_{y^\ell_{ij}}]_{ij}$ with $c_\ell$ a scalar and $[x^\ell_{ij}],[y^\ell_{ij}] \in \M_n(X)$ for each $\ell \in \N$.
\end{theorem}

\begin{proof}
First let us prove the $\le$ inequality.
Suppose that we have a representation  $\mu = \alpha \cdot D \cdot \beta$ as in the statement.
By definition, this means that $\mu$ is the limit in $\M_m(\cF^n(X))$ of the truncated products $\mu_r = \alpha^{rn } \cdot D \cdot \beta^{rn }$ as $r \to \infty$.
By Theorem \ref{Thm-Alternative-N-Lip-Free} we get
\begin{multline*}
\n{\mu_r}_{\M_m(\cF^n(X))} \le \n{\alpha^{rn }} \n{\beta^{rn }} \sup_{1 \le \ell \le r} |c_\ell| \n{[x^\ell_{ij} - y^\ell_{ij}]} \\
\le \n{\alpha} \n{\beta} \sup_{\ell} |c_\ell| \n{[x^\ell_{ij} - y^\ell_{ij}]}    
\end{multline*}
so taking the limit as $r\to\infty$ yields the $\le$ inequality.

Assume now without loss of generality that $\mu\not=0$ and $\n{\mu}_{\M_m(\cF^n(X))}<1$.
Then there exists a sequence $\{\mu_s\}_{s=1}^\infty$ of nonzero terms in $\M_m( \Span\{\delta_x\}_{x\in X} )$ such that
\[
\mu = \sum_{s=1}^\infty \mu_s \qquad \text{and} \qquad \sum_{s=1}^\infty \n{\mu_s}_{\M_m(\cF^n(X))} < 1.
\]
Let $\varepsilon$ be a number with $0 < \varepsilon < 1 -\sum_{s=1}^\infty \n{\mu_s}_{\M_m(\cF^n(X))}$.
For each $s$, Proposition \ref{Thm-Alternative-N-Lip-Free} gives  a representation $\mu_s = \alpha_s \cdot D_s \cdot \beta_s$, 
with  $\alpha_s \in \M_{m,N_sn}$, $\beta \in \M_{N_sn,m}$, and $D_s \in \M_{N_s}(\M_n(X))$   a diagonal matrix whose diagonal entries are $\big[\delta_{x^{\ell,s}_{ij}} - \delta_{y^{\ell,s}_{ij}}\big]_{ij}$, $1 \le \ell \le N_s$ and such that
\[
\n{\alpha_s} \n{\beta_s} \max_{1 \le \ell \le N_s} \n{[x^{\ell,s}_{ij} - y^{\ell,s}_{ij}]} < \n{\mu_s}_{\M_m(\cF^n(X))} + \varepsilon/2^s.
\]
Observe that $\lambda_s=\max_{1 \le \ell \le N_s} \n{[x^{\ell,s}_{ij} - y^{\ell,s}_{ij}]}>0$ because $\mu_s \not=0$.
Define an infinite diagonal matrix $D$ of $n \times n$ blocks whose entries are the concatenation of the finite sequences
\[
\lambda_s^{-1} [\delta_{x^{\ell,s}_{ij}} - \delta_{y^{\ell,s}_{ij}}]_{ij}, \quad 1 \le \ell \le N_s.
\]
Observe that   $D \in \M_{\infty}(\M_n(X))$ and 
\[
\n{D} = \sup_{s \in \N, 1 \le \ell \le N_s} \lambda_s^{-1}\n{[x^{\ell,s}_{ij} - y^{\ell,s}_{ij}]} = 1.
\]
Note that by rescaling, we may assume
\[
\n{\alpha_s} = \n{\beta_s} < \bigg( \frac{ \n{\mu_s}_{\M_m(\cF^n(X))} + \varepsilon/2^s }{ \lambda_s } \bigg)^{1/2}
\]
Now define scalar matrices
\[
\alpha =
\begin{bmatrix}
\lambda_1^{1/2}\alpha_1 &\lambda_2^{1/2}\alpha_2 &\lambda_3^{1/2}\alpha_3 & \cdots\\
\end{bmatrix}
\in \M_{m,\infty}
\]
and
\[
\beta = \begin{bmatrix}
\lambda_1^{1/2}\beta_1 &\lambda_2^{1/2}\beta_2 & &\lambda_3^{1/2}\beta_3 &\cdots\\
\end{bmatrix}^T \in \M_{\infty,m}.
\]
Observe that $\n{\alpha}, \n{\beta}<1$, and $\mu = \alpha \cdot D \cdot \beta$.
\end{proof}

\section{Differentiation in operator spaces}\label{SectionDiff}

We now  discuss Gateaux $\R$-differentiability of maps between operator spaces.   We point out that we will make use of classic results on differentiability of Lipschitz maps between Banach spaces (\cite{Mankiewicz1973}), and those results are only valid for Banach spaces over the reals.   

For the definition of Gateaux $\R$-differentiability, see Definition \ref{DefinitionGateauxDiff} above. We now recall its weak$^*$ version (cf.~\cite[Definition 14.2.17]{AlbiacKaltonBook}):

\begin{definition}\label{DefinitionGateauxDiffWeak}
Let $X$ and $Y$ be $\R$-Banach spaces. A map $f:X\to Y^*$  is \emph{ Gateaux $w^*$-$\R$-differentiable at $x\in X$} if for all $a\in X$ the limit 
\[D^*f_x(a)= w^*\dash\lim_{\lambda\to 0}\frac{f(x+\lambda a)-f(x)}{\lambda}\]
exists and the map $a\in X\mapsto D^*f_x(a)\in Y^*$ is $\R$-linear and  bounded.

\end{definition}

Given an  $\R$-Banach space $X$, we need an idea of ``almost everywhere'' $\R$-differentiability on $X$. If $X$ has finite dimension, the Lebesgue measure does the job: given  a linear isomorphism between  $X$ and $\R^{\mathrm{dim}(X)}$, the Lebesgue measure on $\R^{\mathrm{dim}(X)}$ induces  via this isomorphism a measure $\mu$ on $X$. A measure of this type  is called \emph{a Lebesgue measure on $X$}, and the notion of a  Borel subset $A\subset X$ having \emph{positive measure} (resp.  \emph{full measure}, \emph{finite measure}, or \emph{zero measure}) is independent of the isomorphism, hence well defined.

Since there is no Lebesgue measure on an infinite dimensional Banach space, the infinite dimensional case requires something different. For that, we use the notion of Haar null sets \cite[Definition 14.2.7]{AlbiacKaltonBook}.

\begin{definition} 
Let $X$ be a separable  $\R$-Banach  space and let $A\subset X$ be a Borel subset.
\begin{enumerate}
\item The set $A$ is  called \emph{Haar null} if there exists a Borel probability measure on $X$ such that $\mu(A+x)=0$ for all $x\in X$. 
\item If $A^\complement $ is Haar null, then $A$ is said to be \emph{Haar full}.
\end{enumerate}
\end{definition}

Notice that, for finite dimensional $\R$-Banach spaces, all the notions of ``null'' sets presented above coincide. Precisely, if $X$ has finite dimension,  a subset $A\subset X$ has measure zero with respect to a Lebesgue measure on $X$ if and only if $A$ is  Haar null \cite[Lemma 14.2.9]{AlbiacKaltonBook}.

The following well-known theorem will be essential for our goals. 

\begin{theorem}\label{ThmDIffInHaarNullSetBanachSp}
Let $X$ and $Y$ be $\R$-Banach spaces, and  $f:X\to Y$ be a Lipschitz map. If $X$ is separable, the following holds.
\begin{enumerate}
\item \emph{(}\cite[Theorem 4.5]{Mankiewicz1973}\emph{)}  If  $Y$ has the Radon-Nikodym property, then $f$ is Gateaux $\R$-differentiable at a Haar full subset of $ X$.

\item \emph{(}\cite[Theorem 3.2]{HeinrichMankiewicz1982}\emph{)} If $Y$ is the dual of a separable Banach space, then   $f$ is Gateaux $w^*$-$\R$-differentiable at a Haar full subset of   $ X$.
\end{enumerate}
\end{theorem}

\begin{proposition}\label{PropComplRadeThm}
Let $X$  and  $Y$ be   $\R$-operator spaces and $f:X\to Y$ be a   Lipschitz map. Then, given $n\in\N$, 
\begin{enumerate}
\item\label{ItemDFBound} $\|Df_x\|_n\leq \Lip_n(f)$ for all $x\in X $ such that $Df_x$ exists, and
\item\label{ItemD*FBound} If $Y$ is a dual space, then $\|D^*f_x\|_{n}\leq \Lip_n(f)$ for all $x\in X $ such that $D^*f_x$ exists.
\end{enumerate}
Moreover, if $f$ is a   Lipschitz embedding, then
\begin{enumerate}\setcounter{enumi}{2}
\item\label{ItemDFInvBound} $\|(Df_x)^{-1}\|_{n}\leq \Lip_n(f^{-1})$ for all $x\in X $ such that $Df_x$ exists.
\end{enumerate}
\end{proposition}

\begin{proof}
\eqref{ItemDFBound} and \eqref{ItemDFInvBound}
Fix $x\in X$ so that $Df_x$ exists. Given $[a_{ij}]\in \M_n(X)$,    we have
\begin{align*}
\|Df_x([a_{ij}])\|_{\M_n(Y)} &= \lim_{\lambda\to 0}\Big\|\Big[  \frac{f(x+\lambda a_{ij})-f(x)}{\lambda}\Big]\Big\| _n\\
&= \lim_{\lambda\to 0}\frac{1}{\lambda}\cdot \Big\|  f_n([x+\lambda a_{ij}])-f_n([x])\Big\|_n.
\end{align*}
Since 
\[|\lambda| \frac{\|[a_{ij}]\|_{\M_n(X)}}{\Lip_{n}(f^{-1})}\leq \Big\|  f_n([x+\lambda a_{ij}])-f_n([x])\Big\|_n\leq|\lambda|\Lip_n(f) \|[a_{ij}]\|_{\M_n(X)},\]
it follows that $\|Df_x\|_n\leq \Lip_n(f)$ and   $\|(Df_x)^{-1}\|_{n}\leq \Lip_n(f^{-1})$.

\eqref{ItemD*FBound} Say $Y=Z^*$. Fix $x\in X$ so that $D^*f_x$ exists. Fix   $[a_{ij}]\in \M_n(X)$. So
$D^*f_x([a_{ij}])\in \M_n(Z^*)=\CB(Z,\M_n)$, and we have
\begin{align*}
\|D^*&f_x([a_{ij}])\|_{ \M_n(Z^*)}\\
 &=\sup\Big\{\|[D^*f_x(a_{ij})(b_{pq})]\|_{nk}  \mid k\in\N, [b_{pq}]\in B_{\M_k(Z)}\Big\}    \\
&=\sup\Big\{\lim_{\lambda\to 0}\Big\|\Big[\Big\langle b_{pq}, \frac{f(x+\lambda a_{ij})-f(x)}{\lambda}\Big\rangle\Big]\Big\| _{nk} \mid k\in\N, [b_{pq}]\in B_{\M_k(Z)}\Big\}.  
\end{align*}
If $[b_{pq}]\in \M_k(Z)$, the map $\mathbf b:Z^*\to \M_k$ given by $b(y)=[\langle b_{pq}, y\rangle]$ for all $y\in Z^*$ is a linear map, so
\begin{align*}
\|D^*f_x([a_{ij}])\|_{ \M_n(Z^*)}
&\leq \lim_{\lambda\to 0}\Big\|\Big[ \frac{f(x+\lambda a_{ij})-f(x)}{\lambda}\Big]\Big\|_n\\
&=\lim_{\lambda\to 0}\frac{1}{\lambda}\cdot \Big\|  f_n([x+\lambda a_{ij}])-f_n([x])\Big\|_n\\
&\leq \Lip_n(f)\cdot \|[a_{ij}]\|_n.  
\end{align*}
Hence, $\|D^*f_x\|_{n}\leq  \Lip_{n}(f)$.
\end{proof}

If $f:X\to Y^*$ is a Lipschitz embedding, we would like to obtain   an upper bound for $\|(D^*f_x)^{-1}\|_{n}$ similarly as we did in Proposition \ref{PropComplRadeThm}\eqref{ItemDFInvBound}. Unfortunately, the norm of $Y^*$ is only lower semicontinuous, so the arguments above are not enough to give us that $\|D^*f_x([a_{ij}])\|_{\M_n(Y^*)}\geq (\Lip_n(f^{-1}))^{-1}$ for all $[a_{ij}]\in \partial B_{\M_n(X)}$. A more sophisticated argument is needed.

\begin{proposition}\label{PropBorelSigmaAlgDual}
Let $X$ be a separable $\R$-Banach space and  $Y\subset X^*$ be a separable subspace. Then the Borel $\sigma$-algebra on $Y$  generated by the norm open subsets of $Y$ and the Borel $\sigma$-algebra generated by the weak$^*$ open subsets of $Y$ are the same.
\end{proposition}

\begin{proof}
Clearly, every weak$^*$ Borel subset of $Y$ is norm Borel. Since $Y$ is separable, every open subset of it is the countable union of open balls in $Y$. Since open balls are countable unions of closed balls, we only need to show that $B_{Y}$ is weak$^*$ Borel in $Y$. Let $(x_n)_n$ be a dense sequence in $B_X$, and for each $n\in\N$ let $V_n=\{x^*\in X^*\mid |x^*(x_n)|\leq 1\}$. Each $V_n$ is weak$^*$ closed in $X^*$, hence weak$^*$ Borel in $X^*$.  Since $B_{Y}=Y\cap \bigcap_nV_n$, it follows that $B_{Y}$ is weak$^*$ Borel in $Y$.
\end{proof}

We can now prove the main tool we need in order to obtain Theorem \ref{ThmCompLipEmbDualSpImpliesCompEmbBOUND:COROLLARY}. The following is an adaptation of   \cite[Lemma 3.1]{HeinrichMankiewicz1982} to the operator space setting; we point out that its main difference from the classic setting lies in  Subclaim \ref{SubclaimContOnF}. 

 \begin{proposition} \label{PropD*FInvBound}
Let $X$ be a separable $\R$-operator space, $Y$ be  a separable   $\R$-operator space, $n\in\N$, and $f:X\to Y^*$ be an $n$-Lipschitz embedding. Then   $\|(D^*f_x)^{-1}\|_{n}\leq \Lip_n(f^{-1})$  at a Haar full subset of   $  X $.
\end{proposition}

\begin{proof}
Let $W=\{x\in X\mid f\text{ is Gateaux }\R\text{-differentiable at }x\}$. By  Theorem \ref{ThmDIffInHaarNullSetBanachSp}, $ W$ is Haar full. We first prove the proposition with the further assumption that   $X$ has finite dimension, in particular,  $X$ has a Lebesgue measure  and $W$ has full Lebesgue measure.	   To simplify notation, let $s=(\Lip_{n}(f^{-1}))^{-1}$.   Given  $[a_{ij}]\in \partial B_{\M_n(X)}$, let
\[Q_{[a_{ij}]}=\Big\{x\in W\mid \|D^*f_x([a_{ij}])\|_{\M_n(Y)}\geq s\Big\}.\]

\begin{claim}
For all $[a_{ij}]\in \partial B_{\M_n(X)}$, the set $W\setminus Q_{[a_{ij}]}$ has measure zero.
\end{claim}

\begin{proof}
Suppose not and fix an offender, say  $[a_{ij}]\in\partial B_{\M_n(X)}$. Pick $m\in\N$ such that 
\[N=\Big\{x\in W\mid \|D^*f_x([a_{ij}])\|_{\M_n(Y)}\leq s\Big(1-\frac{1}{m}\Big)\Big\}\]
has positive measure.

\begin{subclaim}\label{SubclaimContOnF}
There exists  $F\subset N$ with positive measure  so that    the map 
\[\Theta_{ij}:x\in F\mapsto D^*f_x(a_{ij})\in Y^*\] is continuous for all $i,j\in \{1,\ldots, n\}$.
\end{subclaim}

\begin{proof}
Consider the map 	  
\[\Lambda:x\in N\mapsto [D^*f_x(a_{ij})]\in   \M_n(Y^*).\] Since $\Lambda$ is given by a derivative, by endowing $\M_n(Y^*)$ with the topology given by the weak$^*$ topology of $Y^*$, it follows that $\Lambda$ is  a Baire class 1 function, i.e., $\Lambda$ is the pointwise (weak$^*$) limit of continuous functions.	In particular,   $\Lambda$ is measurable with respect to the Borel $\sigma$-algebra of $\M_n(Y^*)$ generated by the  weak$^*$ topology. Since $X$ is separable, there exists a separable   subspace $Z\subset Y^*$ such that $\mathrm{Im}(\Lambda)\subset Z$. Let $\|\cdot\|_\infty$ be the supremum  norm on $\M_n(Y^*)$, i.e., $\|[b_{ij}]\|_\infty=\max_{i,j}\|b_{ij}\|$ for all $[b_{ij}]\in \M_n(Y^*)$.  Proposition \ref{PropBorelSigmaAlgDual} implies that $\Lambda$ is also measurable as a map into $(\M_n(Z),\|\cdot\|_\infty)$. 

Fix $r>0$ such that $(r\cdot B_{X})\cap N$ has positive (finite) measure.  A theorem of Lusin (see \cite[Theorem 17.12]{KechrisBook1995})  gives us an $F\subset (r\cdot B_{X})\cap N$ with positive measure so that $\Lambda\restriction F:F\to (\M_n(Z),\|\cdot\|_\infty)$ is continuous. Since the coordinate projections  $\pi_{ij}:\M_n(Z)\to Y^*$ are continuous, we are done. 
\end{proof}

Let $F\subset N$ be   given by Subclaim \ref{SubclaimContOnF}. Given $x\in X$ and $i,j\in \{1,\ldots, n\}$, define a map $\psi_{x,i,j}:\R\to \R$ by 
\[\psi_{x,i,j}(t)=\chi_{F}(x+ta_{ij})\]
for all $t\in\R$, where $\chi_F$ is the characteristic function of $F$.

\begin{subclaim}\label{Subclaim}
Let \[E=\bigcap_{i,j\in \{1,\ldots,n\}}\Big\{x\in N\mid 0\text{ is a Lebesgue point of }\psi_{x,i,j}\Big\}.\] Then $N\setminus E$ has measure zero.
\end{subclaim}

\begin{proof}
Let $i,j\in\{1,\ldots, n\}$, $x_0\in X$, and $L=\{x_0+ta_{ij}\mid t\in \R\}$.  Since for an integrable map almost every point is a Lebesgue point, letting 
  \[E_{i,j,x_0}=\Big\{x_0+ta_{ij}\in N\cap L\mid t\text{ is a Lebesgue point of }\psi_{x_0,i,j}\Big\},\]
we have   that $  L\setminus E_{i,j,x_0}$ has measure zero as a subset of $L$.

 For each $i,j\in \{1,\ldots,n\}$, let 
\[E_{i,j}=\Big\{x\in N\mid 0\text{ is a Lebesgue point of }\psi_{x,i,j}\Big\}.\]
Notice that $0$ is a Lebesgue point of $\psi_{x,i,j}$ if and only if $t$ is a Lebesgue point of $\psi_{x_0,i,j}$, where $x=x_0+ta_{ij}$. Therefore, the previous paragraph and Fubini's theorem imply that $  L\setminus E_{i,j}$ has measure zero as a subset of $L$ for all lines $L\subset X$ in the direction of $a_{ij}$. Hence $N\setminus E_{i,j}$ has measure zero   for all $i,j\in \{1,\ldots, n\}$, which implies that  $N\setminus E$ has measure zero.
\end{proof}

By  Subclaim \ref{Subclaim}, we can pick $x\in F\cap E$.  By  our choice of $F$, there exists $\delta>0$ such that for all $i,j\in \{1,\ldots, n\}$ we have
\[\Big \|D^*f_x(a_{ij})-D^*f_{x+y}(a_{ij})\Big\|<\frac{s}{4mn^2}\]
for all $y\in X$ with $\|y\|\leq \delta$ and $x+y\in F$. In particular,
\begin{align*}
\|[D^*&f_{x+ta_{ij}}(a_{ij})]\|_{\M_n(Y^*)}\\
&\leq \|[D^*f_{x}(a_{ij})]\|_{\M_n(Y^*)}+ \|[D^*f_{x}(a_{ij})]-[D^*f_{x+ta_{ij}}(a_{ij})]\|_{\M_n(Y^*)}\\
&\leq\|[D^*f_{x}(a_{ij})]\|_{\M_n(Y^*)}+ \sum_{i,j\in \{1,\ldots,n\}}\|D^*f_{x}(a_{ij})-D^*f_{x+ta_{ij}}(a_{ij})\|\\
&\leq s\Big(1-\frac{1}{m}\Big)+\frac{s}{4m}\\
&=s\Big(1-\frac{3}{4m}\Big)
\end{align*}
for all $t\in \R$ with $t\leq \delta$ and $x+ta_{ij}\in F$ for all $i,j\in \{1,\ldots,n\}$.

By the definition of Lebesgue point, there exists $\eps>0$ such that 
  the Lebesgue measure of 
\[A=\bigcap_{i,j\in \{1,\ldots,n\}}\Big\{t\in (0,\eps)\mid  x+ta_{ij}\in F\Big\}\]
is at least $\eps(1-s/(4m\Lip_n(f)))$. By replacing $\eps$ by a smaller positive number if necessary, assume that $\eps\leq \delta$. Let $B=(0,\eps)\setminus A$, so $B$ has measure at most $\eps s/(4m\Lip_n(f))$.

To simplify notation, let $\mathbf b=[f(x+\eps a_{ij})-f(x)]\in \CB(Y,\M_n)$. Since 
\[\|[f(x+\eps a_{ij})-f(x)]\|_n=\|f_n([x+\eps a_{ij}])-f_n([x])\|_n\geq s\eps,\]
there exists $k\in\N$ such that $\|\mathbf b_k\|_k>s\eps(1-1/(2m))$, so we can pick $[b_{pq}]\in B_{\M_k(Y)}$ such that $\|\mathbf b_k([b_{pq}])\|>s\eps(1-1/(2m))$.  Pick $\bar\xi=(\xi_i)_{i=1}^{nk},\bar\zeta=(\zeta_{i})_{i=1}^{nk}\in \R^{\oplus nk}$ such that 
\begin{enumerate}
\item $\sum_{i=1}^{nk}\|\xi_i\|^2=\sum_{i=1}^{nk}\|\zeta_i\|^2=1$, and
\item $\langle\mathbf b_k([b_{pq}])\bar\xi,\bar\zeta\rangle>s\eps(1-1/(2m))$.
\end{enumerate}

Define $\varphi:(0,\eps)\to \R$ by letting
\[\varphi(t)=\Big\langle f_n([x+t a_{ij}])([b_{pq}])\bar\xi,\bar\zeta\Big\rangle\]
for all $t\in (0,\eps)$. Clearly,   $\Lip(\varphi)\leq \Lip_n(f)$. Therefore,  Rademacher's Theorem implies that  $\varphi'(t)$ exists almost everywhere, and we have that $|\varphi'(t)|\leq \Lip_n(f)$ and \[\varphi'(t)=\langle [D^*f_{x+ta_{ij}}(a_{ij})]([b_{pq}])\bar\zeta,\bar \xi\rangle\] for all such $t$.  Hence, we have that
\begin{align*}
s\eps\Big(1-\frac{1}{2m}\Big)&<\varphi(\eps)-\varphi(0)\\
&=\int_{(0,\eps)}\varphi'(t)dt\\
&= \int_{A}\varphi'(t)dt+\int_{B}\varphi'(t)dt\\
&\leq \int_{A}\|[ D^*f_{x+ta_{ij}}(a_{ij})]\|_{\M_n(Y^*)}dt+\frac{s\eps}{4m}\\
&\leq s\eps\Big(1-\frac{3}{4m}\Big)+\frac{s\eps}{4m}\\
&= s\eps\Big(1-\frac{1}{2m}\Big);
\end{align*} 
a contradiction.
\end{proof}

Since $X$ is separable,   we can pick    $D\subset \partial B_{\M_n(X)}$ which is countable and dense. Since each $Q_{[a_{ij}]}$ has full measure, the same holds for
\[Q=\bigcap_{[a_{ij}]\in D}Q_{[a_{ij}]}.\]
By the continuity of each $Df_x:X\to Y^*$, it follows that $\|Df_n([a_{ij}])\|_{\M_n(Y^*)}\geq s$ for all $n\in\N$ and all $[a_{ij}]\in \partial B_{\M_n(X)}$. So $\|(Df_x)^{-1}\|_n\leq s^{-1}$ for all $x\in Q$, and this completes the proof of the proposition in the case $X$ has finite dimension. 

We now deal with the infinite dimensional case. For that, we follow the proof of \cite[Theorem 14.2.19]{AlbiacKaltonBook}. Let $(X_k)_k$ be an increasing sequence of finite dimensional subspaces of $X$ so that $\bigcup_kX_k$ is dense in $X$. For each $k\in\N$, let $D_k$ be the set of all $x\in X$ such that there exists a linear operator $T^k:X_k\to Y^*$ such that
\begin{enumerate}
\item for all $a\in X_k$  and all $y\in Y$ we have \[T^k(a)(y)=\lim_{\lambda\to 0}\frac{f(x+ta)(y)-f(x)(y)}{\lambda},\]
and 
\item $s\|[a_{ij}]\|_{\M_n(X)}\leq \|T^k_n([a_{ij}])\|_{\M_n(Y^*)}$ for all $n\in\N$ and all $[a_{ij}]\in \M_n(X)$.
\end{enumerate}
For each $k\in\N$, let $A_k=X\setminus D_k$.

For each $z\in X$, define $f_z:X\to Y^*$ by $f_z(x)=f(x-z)$. So, for each $k\in\N$, the result for finite dimensional domains implies that the set
\begin{align*}
(z&+ A_k)\cap X_k\\
&=\Big\{x\in X_k\mid f_z\restriction X_k \text{ is not }w^*\text{-Gateaux  differentiable}\\
&\ \ \ \ \ \ \ \ \ \ \ \ \ \ \ \ \ \ \ \ \ \ \ \ \text{ at }x\text{ or }\|(D^*f_z\restriction X_k)^{-1}\|>s^{-1}\Big\}
\end{align*}
has Lebesgue measure zero for all $z\in X$. Hence, \cite[Lemma 14.2.9 and Lemma 14.2.12]{AlbiacKaltonBook} imply that $\bigcup_kA_k$ is Haar null. 

One can easily prove (see \cite[Theorem 14.2.13]{AlbiacKaltonBook}) that $\bigcap_nD_n$ equals the set of all $x\in X$ such that $f$ is Gateaux $w^*$-differentiable at $x$ and $\|(D^*f_z\restriction X_k)^{-1}\|_n\leq s^{-1}$ --- we leave the details to the reader ---, so we are done. 
\end{proof}

We will now use Proposition \ref{PropD*FInvBound} to show that almost complete Lipschitz embeddability of a separable $\K$-operator space into a dual $\K$-operator space implies almost complete $\R$-linear embeddability. But before that, we need a result on   ultrapowers of operator spaces.

The proof of the following proposition is straightforward and we omit it (see also \cite[Sec. 2.8]{Pisier-OS-book}).

\begin{proposition}\label{PropMapUltraOfDualDualOfUltra}
Let $X$ be a $\K$-operator space and $ \cU$ an ultrafilter on some index set $I$. The map $j	:(X^*)^\cU\to (X^\cU)^*$   given by \[j\Big([(x^*(n))_n]\Big)\Big([(x(n))_n]\Big)=\lim_{n, \cU}x^*(n)\Big(x(n)\Big),\] for all $[(x(n))_n]\in X^\cU$ and all $[(x^*(n))_n]\in (X^*)^\cU$, is a complete isometry into $(X^\cU)^*$.\qed
\end{proposition}
 
The following should be compared with \cite[Theorem 2.2]{Stern1978} (cf. \cite[Theorem 1.6]{HeinrichMankiewicz1982}).

\begin{proposition}\label{PropDownLoewenheinSkolenOpSp}
Let $X$ and $Y$ be $\R$-operator spaces with $X\subset Y$ and $X\neq \{0\}$. Then there exists an $\R$-operator space $Z$, a nonprincipal ultrafilter $\cU$ on some index set $\Gamma$, and a surjective complete isometry $v:Z^\cU\to Y^\cU$ such that
\begin{enumerate}
\item $X\subset  Z\subset Y$,
\item $\mathrm{dens}(Z)=\max\{\mathrm{dens}(X), \aleph_0\}$, and, 
\item letting $i: Z\to Y$ be the inclusion map, the following diagram commutes.
\[\xymatrix{ Z^\cU\ar[r]^v& Y^\cU\\
  Z\ar[r]_i\ar[u]^{I_{Z}}& Y\ar[u]_{I_{Y}}}
\]
\end{enumerate}
\end{proposition}

\begin{proof}
This proof uses model theory and we refer the reader to \cite{ChangKeislerBook} for details. Define  $\Q$-operator spaces as the $\Q$-vector subspaces of bounded operators on a Hilbert space. First notice that the class of $\Q$-operator spaces is an axiomatizable class on a countable language (cf. \cite[Section 5]{Stern1978} and \cite[Appendix B]{GoldbringSinclair2015}). Let $X_0\subset X$ be a  dense $\Q$-vector subspace of cardinality $\mathrm{dens}(X)$. By the Downwards L\"{o}ewenheim-Skolem Theorem, there exists a  $\Q$-operator space $Z_0\subset Y$ containing $X_0$ so that $| Z_0|=\max\{|X_0|, \aleph_0\}$ and $Z_0$ is elementarily equivalent to $Y$. Therefore, Keisler-Shelah's isomorphism theorem \cite[Theorem 6.1.15]{ChangKeislerBook} implies that there exists an index set $\Gamma$ and an ultrafilter $\cU$ on $\Gamma$ such that $ Z_0^\cU$ is completely isometric to $Y^\cU$. Let $  Z$ be the closure of $ Z_0$, so $\mathrm{dens}(  Z)=\max\{\mathrm{dens}(X), \aleph_0\}$ and $X\subset  Z\subset Y$. Moreover, it is straightforward to check that $ Z_0^\cU$ and $  Z ^\cU$ are completely isometric, so we are done.
\end{proof}

\begin{theorem}\label{ThmCompLipEmbDualSpImpliesCompEmbBOUND}
Let $X$ and $Y$ be   $\K$-operator spaces, let $n\in\N$, and assume that $X$ is separable.  If $\lambda>0$ and $X$   Lipschitzly embeds into $Y^*$ by  a map with $n$-th distortion at most $\lambda$, then $X$  $\R$-linearly    embeds into $Y^*$ by a map with $n$-th distortion at most $\lambda$.
\end{theorem}

\begin{proof}
Firstly, notice that we can assume $\K=\R$. Indeed, if $\K=\C$, then denote $X$ and $Y$ viewed as $\R$-operator spaces by $X_\R$ and $Y_\R$, respectively. Then we have canonical complete $\R$-isomorphisms $Y^*\cong Y_\R^*\oplus Y^*_\R\cong (Y_\R\oplus Y_\R)^*$. So we assume $\K=\R$ for the remainder of the proof. 

Let $f:X\to Y^*$ be a Lipschitz embedding. We first prove the result with the extra condition that $Y$ is separable, so assume $Y$ is separable. By Theorem \ref{ThmDIffInHaarNullSetBanachSp},  $f$ is Gateaux $w^*$-$\R$-differentiable at a Haar full subset of   $  X$ and $\|Df_x\|_n\leq\Lip_n(f)$ for all $x$'s in this subset. By Proposition \ref{PropD*FInvBound}, $\|(Df_x)^{-1}\|_n\leq\Lip_n(f^{-1})$  in a Haar full subset of  $  X$. Pick $x\in X$ satisfying both conditions. Then $D^*f_x:X\to Y^*$ is an  $\R$-linear isomorphism into a subspace of $Y^*$ whose $n$-th distortion is at most the one of $f$.

We  prove the result for an arbitrary $Y$. 
Since $X$ is separable, so is $\mathrm{Im}(f)$.
Pick a separable subspace $Z_0\subset Y$ which completely norms $ \overline{\mathrm{span}}\{\mathrm{Im}(f)\}$, i.e., for all $n\in\N$ and all $[x^*_{ij}]\in \M_{n}( \overline{\mathrm{span}}\{\mathrm{Im}(f)\})$ we have
\[\|[x^*_{ij}]\|_n=\sup\Big\{\|[x^*_{ij}]([x_{pq}])\|\mid k\in\N,\ [x_{pq}]\in B_{\M_k(Z_0)}\Big\}.\]
 Proposition \ref{PropDownLoewenheinSkolenOpSp} gives a separable $\R$-operator space $Z\subset Y$ containing $Z_0$, an ultrafilter $\cU$ on an index set $\Gamma$, and a complete linear isometry $v:Z^\cU\to Y^\cU$ such that $I_{Y}\circ \iota=v\circ I_{Z}$, where $\iota:Z\to Y$ is the inclusion. Let $j:(Z^*)^\cU\to (Z^\cU)^*$ be the map in Proposition \ref{PropMapUltraOfDualDualOfUltra}, and define $u= I_{Y}^*\circ(v^{-1})^*\circ j\circ I_{Z^*} $. Then $u:Z^*\to Y^*$, and for all $n,k\in\N$, all $[z^*_{ij}]\in \M_n(X^*)$, and all $[z_{pq}]\in \M_k(Y)$ we have that
\begin{align*}
I^*_{Y}\circ(v^{-1})^*\circ j\circ I_{Z^*}([z^*_{ij}])([z_{pq}])&=[j\circ I_{Z^*}(z^*_{ij})(v^{-1}(I_{Y}(z_{pq})))]\\
&=[j\circ I_{Z^*}(z^*_{ij})(I_{Z}(z_{pq}))]\\
&=\lim_{m,\cU}[I_{Z^*}(z^*_{ij})(z_{pq})]\\
&=\lim_{m,\cU}[z^*_{ij}(z_{pq})]\\
&=[z^*_{ij}(z_{pq})]\\
&=[z^*_{ij}]([z_{pq}]).
\end{align*}  
So, $u$ is a complete linear isometry into $Y^*$. Let $r:Y^*\to Z^*$ be the restriction operator, so $r$ is a complete contraction. Since $Z$ completely norms the image of $f$, the map $r\circ f:X\to Z^*$ is a complete Lipschitz embedding. Since $Z$ is separable, the result for duals of separable operator spaces (which we proved in the beginning of this proof) implies that $X$ $\lambda$-completely $\R$-isomorphically embeds into $Z^*$ for all  $\lambda\geq  \Lip_n(f)\cdot\Lip_n(f^{-1})$. Therefore, as $Z^*$ is completely $\R$-linearly isometric to a subspace of $Y^*$, it follows that $X$ completely $\R$-isomorphically embeds into $Y^*$.
\end{proof}

\begin{proof}[Proof of Theorem \ref{ThmCompLipEmbDualSpImpliesCompEmbBOUND:COROLLARY}]
This is an immediate consequence of Theorem  \ref{ThmCompLipEmbDualSpImpliesCompEmbBOUND}.
\end{proof}

 \begin{proof}[Proof of Corollary \ref{CorLipThm}]
This follows from   Theorem \ref{ThmCompLipEmbDualSpImpliesCompEmbBOUND:COROLLARY} and Proposition \ref{PropYoplusYBar}.
\end{proof}

\section*{Conflicts of Interest}

On behalf of all authors, the corresponding author states that there is no conflict of interest.
 
   \newcommand{\etalchar}[1]{$^{#1}$}
\providecommand{\bysame}{\leavevmode\hbox to3em{\hrulefill}\thinspace}
\providecommand{\MR}{\relax\ifhmode\unskip\space\fi MR }
\providecommand{\MRhref}[2]{%
  \href{http://www.ams.org/mathscinet-getitem?mr=#1}{#2}
}
\providecommand{\href}[2]{#2}

\end{document}